\documentclass{amsart}
\usepackage{amssymb}
\usepackage{amsfonts}
\usepackage{amsmath}
\usepackage{paralist}

\newtheorem{theorem}{Theorem}[section]

\newtheorem{lemma}[theorem]{Lemma}
\newtheorem{proposition}{Proposition}[section]

\theoremstyle{definition}

\newtheorem{remark}{Remark}[section]

\numberwithin{equation}{section} \subjclass{Primary: 35K59, 35B45,
35Q92 ; Secondary: 92B05} \keywords{parabolic systems, dynamic
boundary conditions, uniform estimates, porous medium equation,
ecology}
 \email{cgal@fiu.edu}
\input{tcilatex}

\begin{document}
\title[Parabolic Systems]{\textbf{Sup-norm estimates for parabolic systems
with dynamic boundary conditions}}
\author[Ciprian G. Gal]{}
\maketitle

\centerline{\scshape Ciprian G. Gal } \medskip {\footnotesize \
\centerline{Department of Mathematics, Florida International University} %
\centerline{Miami, FL, 33199, USA} }



\begin{abstract}
\noindent {\footnotesize We consider parabolic systems with
nonlinear dynamic boundary conditions, for which we give a
rigorous derivation. Then, \ we give them several physical
interpretations which includes an interpretation for the
porous-medium equation, and for certain reaction-diffusion systems
that occur in mathematical biology and ecology. We devise several
strategies which imply (uniform) }$L^{p}${\footnotesize \ and
}$L^{\infty }${\footnotesize \ estimates on the solutions for the
initial value problems considered.}
\end{abstract}

\section{Introduction}

In this article, we consider the following system of quasilinear parabolic
equations%
\begin{equation}
\partial _{t}u_{i}-\Delta \left( A_{i}\left( u_{i}\right) \right)
+f_{i}\left( x,t,\overrightarrow{u}\right) =0\text{, in }\Omega \times
\left( 0,\infty \right) ,  \label{s1}
\end{equation}%
for $i=1,...,m$, where $\overrightarrow{u}=\left( u_{1},...,u_{m}\right) ,$ $%
\Omega $ is a bounded domain in $\mathbb{R}^{N},$ $N\geq 1$, with
sufficiently smooth boundary $\Gamma :=\partial \Omega $ (which is at least
of class $\mathcal{C}^{2}$), for some given functions $A_{i}$ and $f_{i}.$
Denote by $\mathbb{N}_{m}=\left\{ 1,...,m\right\} $ and consider two
mutually disjoint (possibly empty) subsets $I_{m},J_{m}\subseteq \mathbb{N}%
_{m}$ such that $I_{m}\cup J_{m}=\mathbb{N}_{m}$. Equation (\ref{s1}) is
subject to the following set of boundary conditions%
\begin{equation}
\partial _{\mathbf{n}}u_{i}+h_{i}\left( x,t,\overrightarrow{u}\right) =0,%
\text{ on }\Gamma \times \left( 0,\infty \right) ,\text{ \thinspace }i\in
I_{m}  \label{s2}
\end{equation}%
and%
\begin{equation}
\delta _{i}\partial _{t}u_{i}+\partial _{\mathbf{n}}\left( A_{i}\left(
u_{i}\right) \right) +g_{i}\left( x,t,\overrightarrow{u}\right) =0,\text{ on
}\Gamma \times \left( 0,\infty \right) ,\text{ \thinspace }i\in J_{m},
\label{s3}
\end{equation}%
for some given functions $g_{i}$ and $h_{i}$. Here $\delta _{i}>0$ for $i\in
J_{m}$, and we may assume, without loss of generality, that $\delta _{i}=0,$
for $i\in I_{m}$. The boundary conditions in (\ref{s2})-(\ref{s3}) may be
also mixed, that is, the boundary $\Gamma $ may consists of two disjoint
open subsets $\Gamma _{1}$ and $\Gamma _{2}$ on which the boundary
conditions may be either of Dirichlet type or of the form (\ref{s2}) and (%
\ref{s3}), respectively. Finally, the model (\ref{s1})-(\ref{s3}) could be
also generalized be letting the reaction terms depend on advection, by
allowing the diffusion rates depend also on $x$ and $t$, or in other various
ways. As usual, we equip the system (\ref{s1})-(\ref{s3}) with the initial
conditions%
\begin{equation}
u_{i\mid t=0}=u_{i0}\text{ in }\Omega ,\text{ }u_{i\mid t=0}=v_{i0}\text{ on
}\Gamma ,\text{ }i\in \mathbb{N}_{m},  \label{s4}
\end{equation}%
where in general, we may have $u_{i0\mid \Gamma }\neq v_{i0},$ $i\in \mathbb{%
N}_{m}$ (i.e., if $u_{i0}$ is well-defined in the trace sense).

We aim to give some results which allow to deduce $L^{\infty }$-estimates
for solutions of (\ref{s1})-(\ref{s4}) assuming that some sort of energy
estimate is apriori known in $L^{p}$-norm for some finite $p$. The main tool
will be an iterative argument following a well-known Alikakos-Moser
technique combined with a suitable form of Gronwall's inequality, which can
then be used to prove that the $L^{p}$-$L^{\infty }$ smoothing property
holds for any solutions of the \emph{non-degenerate} parabolic system (\ref%
{s1})-(\ref{s4}) (e.g., at least when $a_{i}\left( u_{i}\right)
:=A_{i}^{^{\prime }}\left( u_{i}\right) $ satisfies (\ref{non}) below). In
order to deal with the full degenerate case (\ref{s1})-(\ref{s4}) (at least
in the case when $a_{i}\left( u_{i}\right) =\left\vert u_{i}\right\vert
^{p_{i}},$ $p_{i}>0$), we employ DeGiorgi's truncation method to prove the $%
L^{p}$-$L^{\infty }$ smoothing property. The precise statements of these
results can be found in Section 3, see Theorems \ref{linf} and \ref{linf3}.
A rigorous derivation and physical interpretation of the system (\ref{s1})-(%
\ref{s4}) shall be given below in Section 2.

Why is it important to establish apriori (possibly, uniform in time) $%
L^{\infty }$-estimates from some given $L^{p}$-estimate? To better give an
idea of our larger scope let us take a look at some history for problems of
the form (\ref{s1})-(\ref{s4}). Problems such as (\ref{s1})-(\ref{s4}) have
already been investigated in a number of papers \cite{Escher, Escher2, CJ,
Ma} assuming that diffusion rates $a_{i}\left( u_{i}\right)
:=A_{i}^{^{\prime }}\left( u_{i}\right) $ satisfy%
\begin{equation}
a_{i}\left( u_{i}\right) \geq d_{i}>0,\text{ for all }u_{i}\in \mathbb{R}%
\text{ and }i\in \mathbb{N}_{m}.  \label{non}
\end{equation}%
For instance, Constantin and Escher \cite{Escher, Escher2} show that unique
(classical) maximal solutions exist in some Bessel potential spaces under
suitable assumptions on the nonlinearities $f_{i},$ $g_{i}$ and $h_{i}$.
Such results also enable the authors to investigate other qualitative
properties concerning global existence and blow-up phenomena (see, also \cite%
{CJ}). These results are also improved by Meyries \cite{Ma}, still in the
non-degenerate case (\ref{non}), by assuming more general boundary
conditions (by also incorporating surface diffusion in (\ref{s3})),\ and by
requiring that the functions $f_{i}\left( \overrightarrow{u}\right) ,$ $%
h_{i}\left( \overrightarrow{u}\right) ,$ $g_{i}\left( \overrightarrow{u}%
\right) $ are dissipative in a certain sense. However, none of these
contributions deal with the degenerate case for equation (\ref{s1}), that
is, when $a_{i}\left( u_{i}\right) $ is allowed to have a polynomial
degeneracy at zero for some (if not all) $i\in \mathbb{N}_{m}$; for
instance, one can take%
\begin{equation}
a_{i}\left( u_{i}\right) =\left\vert u_{i}\right\vert ^{p_{i}},\text{ }%
p_{i}>0.  \label{ex}
\end{equation}%
Moreover, it is well-known in the scalar case $m=1$, that when at least one
of the source terms, the bulk nonlinear term $f_{1}$ or the boundary term $%
g_{1}$ is present in (\ref{s1})-(\ref{s2}), conditions can be derived on
their growth rates which imply either the global existence of solutions or
blow-up in finite time \cite{FQ}. Namely in the non-degenerate case, for $%
\lambda ,\mu \in \left\{ 0,\pm 1\right\} $ with $\max \left\{ \lambda ,\mu
\right\} =1$, $f_{1}\left( s\right) :=-\lambda \left\vert s\right\vert
^{r_{1}-1}s$ and $g_{1}\left( s\right) :=-\mu \left\vert s\right\vert
^{r_{2}-1}s$, solutions of%
\begin{equation}
\partial _{t}u-\nu \Delta u+f_{1}\left( u\right) =h_{1}\left( x\right) ,%
\text{ in }\Omega \times (0,+\infty ),  \label{1.11bb}
\end{equation}%
subject to the dynamic condition%
\begin{equation}
\partial _{t}u+\nu b\partial _{\mathbf{n}}u+g_{1}\left( u\right)
=h_{2}\left( x\right) ,\text{ on }\Gamma \times \left( 0,\infty \right) ,
\label{1.12bb}
\end{equation}%
are globally well-defined, for every given (sufficiently smooth) initial
data (\ref{s4})$,$\ if $r_{1}r_{2}>1$ and $\lambda r_{1}+\mu r_{2}>0$.
Furthermore, \cite{FQ} shows that if we further restrict the growths of $%
r_{1}$,$r_{2}$ so that $r_{1}<\left( N+2\right) /\left( N-2\right) $ and $%
r_{2}\leq N/\left( N-2\right) $, then the global solutions are also bounded.
These restrictions can be eventually removed and more general conditions on $%
f_{1},g_{1}$ can be deduced (see, e.g., \cite{G0}). On the other hand, if $%
\lambda =0$, $\mu =1,$ then some solutions blowup in finite time with blowup
occurring in the $L^{\infty }$-norm at a rate $\left( t-T_{\ast }\right)
^{-\left( r_{2}-1\right) },$ for some additional conditions on $u_{0}$ and $%
r_{2}$. In the same way, when $\mu =0$ and $\lambda =1$, then some solutions
blowup in finite time with a blowup rate which depends on $r_{1}$ and $u_{0}$
(see \cite{Be2}). In the case when both $\mu \in \mathbb{R}$ and $\lambda >0$
are nonzero, blowup may still occur for superlinear growth of $f_{1}$ and
any growth of $g$ (see \cite{G0}). The occurrence of blow up phenomena is
closely related to the blowup problem for the ordinary differential equation
\begin{equation*}
u_{t}+H\left( u\right) =0,
\end{equation*}%
where either $H=f_{1}$ or $H=g_{1}$. More precisely, it is easy to see that
solutions of the ODE are spatially homogeneous solutions of either equation (%
\ref{1.11bb}) or (\ref{1.12bb}), and so if these solutions blowup in finite
time so do the solutions of (\ref{1.11bb})-(\ref{1.12bb}). It is worth
mentioning that in \cite{CJ} a criterion for the global existence of a
(classical) maximal solution (on some interval $[0,t_{+})$)\ to (\ref{1.11bb}%
)-(\ref{1.12bb}) is established using a variation of parameter formula. In
particular, it is shown that if $t_{+}<\infty $ then necessarily we must have%
\begin{equation*}
\underset{t\rightarrow t_{+}}{\lim \sup }\left\Vert u\left( t\right)
\right\Vert _{L^{\infty }\left( \Omega \right) }=\infty .
\end{equation*}%
Therefore, it appears that in order to deduce global existence of classical
solutions to systems of the form (\ref{s1})-(\ref{s4}), (\ref{non}), it is
generally required that we should deduce bounds on the solutions in $%
L^{\infty }$-norm (see \cite{Ma} also).

Finally, the $L^{p}$-$L^{\infty }$ smoothing property also becomes an
essential tool in attractor theory where it can be used to establish the
existence of an absorbing set in $L^{\infty }$-norm if this property can be
deduced easily in $L^{p}$-norm for some finite $p$ (in many applications in
physics and mechanics, $p$ is equal to either $1$ or $2$). Recall that a
subset $\mathcal{B}\subset \mathcal{H}$, where $\mathcal{H}$ is a
topological space endowed with a given metric, is called \emph{absorbing} if
the orbits corresponding to bounded sets $\mathcal{V}$\ of initial data
enter into $\mathcal{B}$ after a certain time (which may depend on the set $%
\mathcal{V}$)\ and will stay there forever. Moreover, we note that in order
to study the long term behavior of the parabolic system (\ref{s1})-(\ref{s2}%
), if the absorbing property holds in $L^{\infty }$-norm, the growth rate of
the nonlinearities $f_{i}$, $g_{i}$ and $h_{i}$ with respect to $u_{i}$
becomes nonessential for further investigations of attractors. Indeed, the
absorbing property can be also established in higher-order $W^{s,p}$-norms
with relative ease provided that it is known in $L^{\infty }$-norm. For the
application of this property to attractor theory for parabolic equations of
the form (\ref{1.11bb}), (\ref{1.12bb}), see \cite{GalNS, G0}, where
explicit dimension estimates for the global attractor for (\ref{1.11bb})-(%
\ref{1.12bb}) are obtained.

The main goal of this paper is to deduce sufficiently general conditions on
the diffusions and sources in (\ref{s1})-(\ref{s3}), which would prevent
blowup of any solution in the $L^{\infty }$-norm, and show that the
parabolic system under consideration is dissipative in a suitable sense. We
outline the plan of the paper, as follows. In Section 2, we give the full
derivation of systems of the form (\ref{s1})-(\ref{s3}), and give physical
interpretations to the dynamic boundary condition (\ref{s2}) for the
porous-medium equation, and some models in ecology. In Section 3, after we
introduce some notations and preliminary facts, we give the statements of
our main results and some further applications. Finally, in Section 4 we
provide the full proofs of these results.

\section{Derivation and interpretation}

Let $\Gamma \subset \mathbb{R}^{N-1}$\ consists of two disjoint open subsets
$\Gamma _{1}$\ and $\Gamma _{2}$, each $\overline{\Gamma }_{i}\backslash
\Gamma _{i}$\ is a $S$-null subset of $\Gamma $\ and $\Gamma =\overline{%
\Gamma }_{1}\cup \overline{\Gamma }_{2}$\ with $\Gamma _{1}\subseteq \Gamma $%
. We shall only give the derivation in the case of the scalar equation%
\begin{equation}
\partial _{t}u-\text{div}\left( a\left( u\right) \nabla u\right) +f\left(
u\right) =h_{1}\left( x\right) ,  \label{PM1}
\end{equation}%
equipped with (nonlinear)\ dynamic boundary conditions%
\begin{equation}
\partial _{t}u+a\left( u\right) \nabla u\cdot \mathbf{n}+g\left( u\right)
=h_{2}\left( x\right) ,  \label{PM2}
\end{equation}%
on $\Gamma _{1},$ and Dirichlet boundary conditions%
\begin{equation}
u_{\mid \Gamma _{2}}=0.  \label{PM3}
\end{equation}%
We can easily extend our arguments to systems as well (see below). Equations
(\ref{PM1})-(\ref{PM3}) are also subject to the initial conditions%
\begin{equation}
u_{\mid t=0}=u_{0}\text{ in }\Omega ,\quad u_{\mid t=0}=v_{0}\text{ on }%
\Gamma .  \label{PM4}
\end{equation}

Standard derivations of the porous medium equation always use the principle
\textquotedblleft amount of fluid in equals amount of fluid
out\textquotedblright\ over a region $\Omega $, occupied by either a liquid
or gas, and is based on the fact that this fluid diffuses from locations of
higher to those of lower pressure. In the traditional approach, the porous
medium\ equation is assumed to hold in the region $\Omega $ and then the
boundary conditions are appended later. There are three standard boundary
conditions that specify the density on the boundary of $\Omega $; they are
Dirichlet and Neumann-Robin type of boundary conditions (see, e.g., \cite{Va}%
). Dynamic boundary conditions for porous medium equations seem to have
appeared before in different contexts \cite{BFP, FL, Su, ST}. For instance,
\cite{BFP} deals with the modelling of the rain water infiltration through
the soil above an aquifer in regimes where there is runoff at the ground
surface. In general, all rain water infiltrates into the soil, but if the
rainfall event is particularly intense, the maximum draining capacity of the
soil is exceeded. In this case, dynamic boundary conditions (see (\ref%
{gen_DBC}) below) are needed to describe the saturation of layers near the
ground surface (cf. also \cite{FL, Su}). Porous-medium like systems (\ref%
{PM1})-(\ref{PM3}) can also be found as part of (larger)\ coupled systems of
partial differential equations (such as, (\ref{s1})-(\ref{s4}))\ describing
the vertical movement of water and salt in a domain splitted in two parts: a
water reservoir and a saturated porous medium below it, in which a
continuous extraction of fresh water takes place (for instance, by the roots
of mangroves) \cite{GV}. Such problems are formulated in terms of equations
for the salt concentration and the water flow in the porous medium, with a
dynamic boundary condition which connects both subdomains. Finally, dynamic
boundary conditions similar to (\ref{PM2}) also appear in certain classes of
parabolic equations with boundary hysteresis (see, e.g., \cite[Section 4]{ST}
and the references therein). For some applications of dynamic boundary
conditions for physiologically structured populations with diffusion we
refer the reader, for instance, to \cite{FH}.

For all the phenomena of the kind discussed here, the method of introducing
dynamic boundary conditions seems ad hoc. It would be more natural if such
boundary conditions could be derived in the context of energy balance and
constitutive laws. Moreover, the usual derivation of the porous medium
equation with standard boundary conditions does not show how to model, for
instance, a water source, which is located on the boundary of $\Omega $. To
this end, we shall rethink the usual derivation of the porous medium
equation (\ref{PM1}), by making essential connections between the
differential equation (\ref{PM1}) and the boundary conditions (\ref{PM2})-(%
\ref{PM3}),\ and, thus, try to convince the reader that our new perspective
is more natural than the traditional way. Let $p\left( x,t\right) $ denote
the pressure of fluid at $x\in \Omega $ and time $t>0$. Consider the mass of
fluid in an element of volume $V$\ given by%
\begin{equation*}
\int_{V}\alpha \left( x\right) u\left( x,t\right) dx,
\end{equation*}%
where $\alpha \left( x\right) >0$ defines the porosity of medium at the
point $x\in \Omega $. Similarly, we define
\begin{equation*}
\int_{\Gamma }\beta \left( x\right) v\left( x,t\right) dS
\end{equation*}%
as the mass of fluid across the surface $\Gamma $, where $\beta \left(
x\right) $ is such that%
\begin{equation*}
\Gamma _{3}:=\left\{ x\in \Gamma _{1}:\beta \left( x\right) >0\right\}
\end{equation*}%
is a set of positive measure and $\Gamma _{3}\subseteq \Gamma _{1}$. In what
follows, we shall take $\Gamma _{2}=\varnothing $ for the sake of
exposition, so that $\Gamma _{1}\equiv \Gamma $. The flux $\mathbb{J}\left(
x,t\right) $, at which the fluid moves across a surface element $S$ with
normal $\mathbf{n}$, is given by
\begin{equation*}
\int_{S}\mathbb{J}\left( x,t\right) \cdot \mathbf{n}dS.
\end{equation*}%
Suppose now there is a source on the boundary $\Gamma $ to be represented by
a function%
\begin{equation*}
\Psi =\Psi (t,x,u,\nabla u).
\end{equation*}%
The amount of fluid leaving the region is still given by $\int_{\Gamma }%
\mathbb{J}\left( x,t\right) \cdot \mathbf{n}dS$, but the amount of fluid
leaving into the region must also take into account the action of the source
$\Psi $\ on $\Gamma $. It is worth pointing out that, in practice, when
rainfall only partially infiltrates the soil, the water will accumulate on
the ground surface $\Gamma _{1}$ as the surface layer becomes saturated;
hence, necessarily, $\Psi \neq 0$ on $\Gamma $. We use the measure space $%
\left( \overline{\Omega },d\mu \right) $ which we redefine as $\left( \Omega
,dx\right) \oplus \left( \Gamma ,dS\right) $. Then the conservation of fluid
in $\overline{\Omega }$\ takes the form%
\begin{align}
& \partial _{t}\left( \int_{\Omega }\alpha \left( x\right) \rho
dx+\int_{\Gamma }\beta \left( x\right) \eta dS\right) +\int_{\Gamma }\mathbb{%
J}\cdot \mathbf{n}dS  \label{deriv} \\
& =\int_{\Omega }\Xi dx+\int_{\Gamma }\Psi dS,  \notag
\end{align}%
where $\Xi =\Xi \left( x,t,u\right) $ denotes any volume source density
function. Notice that equation (\ref{deriv}) must also account for a term
like
\begin{equation*}
\int_{\Gamma }\beta \left( x\right) vdS,
\end{equation*}%
due to the presence of the source density $\Psi $ at $\Gamma $. Assuming
that $\mathbb{J}$ is sufficiently smooth and applying the divergence theorem
in (\ref{deriv}), we deduce%
\begin{align}
& \partial _{t}\left( \int_{\Omega }\alpha \left( x\right) udx+\int_{\Gamma
}\beta \left( x\right) udS\right) +\int_{\Omega }div\left( \mathbb{J}\right)
dx  \label{deriv2} \\
& =\int_{\Omega }\Xi dx+\int_{\Gamma }\Psi dS.  \notag
\end{align}%
Assuming that the density functions $u,v$ are also differentiable with
respect to $t>0$ and since (\ref{deriv2}) holds for any subdomain $\Omega
_{0}\subseteq \Omega $, the usual argument yields the following differential
equation%
\begin{equation}
\partial _{t}\left( \alpha \left( x\right) u\left( x,t\right) \right) =-%
\text{div}\left( \mathbb{J}\left( x,t\right) \right) +\Xi \left( x,t,u\left(
x,t\right) \right) \text{, }x\in \Omega ,\text{ }t>0.  \label{gen_PM}
\end{equation}%
Then, from (\ref{deriv2}) the boundary condition becomes%
\begin{equation*}
\int_{\Gamma }\left[ \partial _{t}\left( \beta \left( x\right) u\left(
x,t\right) \right) -\Psi \left( x,t,u\left( x,t\right) ,\nabla u\left(
x,t\right) \right) \right] dS=0,\text{ }t>0,
\end{equation*}%
which clearly holds if%
\begin{equation}
\partial _{t}\left( \beta \left( x\right) u\left( x,t\right) \right) -\Psi
\left( x,t,u\left( x,t\right) ,\nabla u\left( x,t\right) \right) =0\text{,
for }x\in \Gamma ,\text{ }t>0.  \label{gen_DBC}
\end{equation}%
Darcy's law states that the flux $\mathbb{J}$ depends on the pressure
gradient so it takes the form%
\begin{equation}
\mathbb{J}\left( x,t\right) =-\frac{\mathcal{K}\left( x\right) }{\nu }%
u\left( x,t\right) \nabla p\left( x,t\right) ,\text{ }x\in \Omega ,\text{ }%
t>0,  \label{Darcy}
\end{equation}%
where $\nu >0$ is the viscosity of the fluid and $\mathcal{K}\left( x\right)
$ defines the permeability of the porous medium. Finally, if one also makes
the assumption that the pressure $p$ is described by an equation of state
involving the density, $p=b\left( u\right) $, then substituting the
appropriate quantities in (\ref{gen_PM}), we obtain the following porous
medium equation%
\begin{equation}
\partial _{t}\left( \alpha \left( x\right) u\left( x,t\right) \right) =\frac{%
1}{\nu }\text{div}\left( \mathcal{K}\left( x\right) a\left( u\left(
x,t\right) \right) \nabla u\left( x,t\right) \right) +\Xi \left( x,t,u\left(
x,t\right) \right) ,\text{ in }\Omega ,\text{ }t>0,  \label{gen_PM2}
\end{equation}%
where we have set $a\left( t\right) \equiv tb\left( t\right) $. The function
$b$ that relates the density to pressure is, in general, monotone and, in
fact, strictly increasing in many applications of the type of fluid being
considered in the literature.

Finally, let us now focus on the boundary condition (\ref{gen_DBC}). We now
show that a quite large class of boundary conditions for equation (\ref{PM1}%
) can be written in this way for various choices of $\Psi $. We emphasize
that in this formulation the boundary conditions arise naturally in the
formulation of the problem. Suppose first $\Gamma _{3}\equiv \Gamma _{1}$
(i.e., $\beta \left( x\right) >0$ a.e. in $\Gamma _{1}$) and $\beta \in
C^{1}\left( \Gamma _{1}\right) .$ Choosing $\Psi \equiv 0$ (i.e., no source
is located at $\Gamma _{1}$), so that $\partial _{t}u\equiv 0$ on $\Gamma
_{1}=\Gamma ,$ therefore%
\begin{equation}
u\left( x,t\right) =u_{0}\left( x\right) ,  \label{D}
\end{equation}%
for $x\in \Gamma _{1}$ and $t\geq 0,$ where $u_{0}$ is the initial condition
associated with equation (\ref{PM1}). Thus, we obtain a Dirichlet boundary
condition for $u$. In order to derive an inhomogeneous Neumann boundary
condition, we suppose that $\Psi $ only depends on $t$. Then, if $u$ is
sufficiently regular, we have $\partial _{t}u(x,t)=\left( 1/\beta \left(
x\right) \right) \Psi \left( t\right) $ on $\Gamma _{1},$ for any $t>0$.
Hence, if $\Gamma _{1}$ is smooth enough as well, we have $\partial
_{t}\left( \nabla u\right) =\nabla \left( \partial _{t}u\right) =\mathbf{%
\gamma }\left( x\right) \Psi \left( t\right) $ on $\overline{\Omega }\times
(0,+\infty ),$ for some $\mathbf{\gamma }\left( x\right) \in \mathbb{R}^{N}$%
. This entails that $\nabla u\left( x,t\right) =\mathbf{H}\left( x,t\right) $
holds for $(x,t)\in \Gamma _{1}\times (0,+\infty )$ and some smooth function
$\mathbf{H}\in \mathbb{R}^{N}.$ Therefore, we have
\begin{equation}
\nabla u\left( x,t\right) \cdot \mathbf{n}=\mathbf{H}\left( x,t\right) \cdot
\mathbf{n},\quad (x,t)\in \Gamma _{1}\times (0,+\infty ).  \label{N}
\end{equation}%
To obtain a Robin boundary condition, we set $\Psi \left( x,t,u,\nabla
u\right) =\beta \left( x\right) e^{Cr}q\left( t\right) ,$ for some $C\in
\mathbb{R}$, where $r$ is defined as the parameter describing the line $\ell
$ which passes through $x$ and contains ${\mathbf{n}}$ such that $r>0$ at
all points on $\ell \cap \Omega $ which are close to $x.$ We thus have $%
\partial _{t}u\left( x,t\right) =e^{Cr}q\left( t\right) $ which implies
\begin{equation*}
\nabla \left( \partial _{t}u\left( x,t\right) \right) \cdot \mathbf{n}\left(
x,t\right) =\partial _{t}\left( \nabla u\left( x,t\right) \cdot \mathbf{n}%
\right) \left( x,t\right) =Ce^{Cr}q\left( t\right) ,
\end{equation*}%
for $(x,t)\in \Gamma _{1}\times (0,+\infty )$. Therefore, we infer%
\begin{equation*}
\partial _{t}\left( \nabla u\left( x,t\right) \cdot \mathbf{n}\right) \left(
x,t\right) -C\partial _{t}u\left( x,t\right) =0,\qquad \text{ on }\Gamma
_{1}\times (0,+\infty ),
\end{equation*}%
so that
\begin{equation}
\nabla u\left( x,t\right) \cdot \mathbf{n}-Cu(x,t)=j\left( x\right) ,\quad
(x,t)\in \Gamma _{1}\times (0,+\infty ).  \label{R}
\end{equation}%
We have thus recovered the most common boundary conditions. On the other
hand, in order to model a water source placed on the boundary $\Gamma _{1},$
we may assume that $\Psi $ only depends nonlinearly on the flux $\nabla
u\cdot \mathbf{n}$ across the boundary as well as on a nonlinear source $%
g\left( u\right) ,$ which may represent effects of reaction or absorption at
$\Gamma _{1}$. That is, we let
\begin{equation*}
\Psi \left( x,u,\nabla u\right) =-a\left( u\right) \nabla u\cdot \mathbf{n}%
-g\left( u\right) -h_{2}\left( x\right) .
\end{equation*}%
Then, the resulting boundary condition for equation (\ref{PM1}) becomes%
\begin{equation}
\beta \left( x\right) \partial _{t}u\left( x,t\right) +a\left( u\left(
x,t\right) \right) \nabla u\left( x,t\right) \cdot \mathbf{n}+g\left(
u\left( x,t\right) \right) =h_{2}\left( x\right) ,  \label{PM2bis}
\end{equation}%
for $\left( x,t\right) \in \Gamma _{1}\times (0,+\infty )$. This is a
reasonably general condition which contains the usual (homogeneous) ones
along with the so-called dynamic boundary condition when $\Gamma _{3}=\Gamma
_{1}\equiv \Gamma $. When ponding or surface runoff occurs at the surface $%
\Gamma _{1}$, it also includes the dynamic boundary condition contained in
Filo-Luckhaus \cite{FL}. Finally, we notice that $\beta \left( x\right) $
may be such that $\Gamma _{3}\neq \Gamma _{1}$, so that we can also cover
the case where the boundary conditions (\ref{PM2bis}) are dynamic only on a
part of the boundary of $\Gamma _{1}$. If the source $\Psi $ also depends on
$v$ in the tangential coordinates of the boundary, i.e., $\Psi =\Psi \left(
x,u,\nabla u,\nabla _{\Gamma _{1}}u\right) ,$ then we can also model local
diffusion in (\ref{PM2bis}) by incorporating the elliptic Laplace-Beltrami
operator $\Delta _{\Gamma _{1}}$ (or other nonlinear differential operators)
on the manifold $\Gamma _{1}\subset \mathbb{R}^{N-1}$.

We will now give a physical interpretation of the effect of a water source
on the patch $\Gamma _{1}$, at least in some cases. A similar approach was
also used in the derivation of heat and wave equations (see, e.g., \cite{Gi}%
). We will mainly focus on the following boundary condition%
\begin{equation}
\partial _{t}u+a\left( u\right) \nabla u\cdot \mathbf{n}=0,\text{ on }\Gamma
_{1}\times \left( 0,\infty \right) .  \label{DYN0}
\end{equation}%
We work in an infinitesimal region on the boundary. Choose a point $x\in
\Gamma _{1}$ and let $B_{\kappa }\left( x\right) $ be a ball of radius $%
\kappa >0$ about $x$. Since $\Gamma _{1}$ is regular, we can choose a
coordinate system for $B_{\kappa }\left( x\right) \cap \overline{\Omega }$
so that the boundary of $B_{\kappa }\left( x\right) \cap \overline{\Omega }$
in the transformed coordinate system is flat, $x$ is mapped to $\overline{x}%
=\left( x_{1},x_{2},...,x_{N-1},0\right) $, that is, the boundary $\Gamma $,
at least locally near $x$ lies on the hyperplane $x_{N}=0$. Then the outward
unit normal $\mathbf{n}$ to $\Gamma $ at $x$ is the unit vector in the
direction of $e_{N}$ which we will denote by $r$. Then, locally near $x,$ (%
\ref{DYN0}) becomes%
\begin{equation}
\partial _{t}u+\partial _{r}\left( A\left( u\right) \right) =0,\text{ }%
\left( r,t\right) \in \left( 0,r_{0}\right) \times \left( 0,t_{0}\right) ,
\label{CLaw}
\end{equation}%
for some sufficiently small positive constants $r_{0}$, $t_{0}$. Observe
that (\ref{CLaw}) resembles nothing more than a scalar \emph{conservation law%
}, where we have set $A\left( u\right) =\int_{0}^{u}a\left( t\right) dt$.
Equation (\ref{CLaw}), subject to initial condition $u\left( r,0\right)
=v_{0}\left( r\right) ,$ $r\in \left( 0,r_{0}\right) ,$ possesses
interesting types of fundamental solutions such as, travelling waves that
describe the movement of a mass of fluid in the direction of the unit normal
$\mathbf{n}\in \mathbb{R}^{N}$, and source-type solutions starting from a
finite mass concentrated at a single point of space, say, $v_{0}\left(
r\right) =\mathcal{C}\delta _{0}\left( r\right) ,$ $\mathcal{C}>0$.

In the latter case, explicit self-similar solutions of the form $%
u(r,t)=\Theta \left( r/t\right) $ are well-known to exist, as a consequence
of the invariance of (\ref{CLaw}) under the scaling $\left( r,t\right)
\longmapsto \left( \lambda r,\lambda t\right) ,$ $\lambda \in \mathbb{R}$.
More precisely, from (\ref{CLaw}) $\Theta $ clearly has to satisfy%
\begin{equation*}
\zeta \Theta ^{^{\prime }}\left( \zeta \right) -\left( A\left( \Theta \left(
\zeta \right) \right) \right) ^{^{\prime }}=0
\end{equation*}%
and this yields, formally, to%
\begin{equation*}
u\left( r,t\right) =\Theta \left( r/t\right) =(A^{^{\prime }})^{-1}\left(
r/t\right) ,
\end{equation*}%
as long as, $(A^{^{\prime }})^{-1}$ is well defined at least in a
sufficiently small real interval. In the former case, one can search for
particular solutions of (\ref{CLaw}) in the form $u\left( r,t\right) =\eta
\left( r-ct\right) ,$ where $c\in \mathbb{R}$ is the speed of the travelling
wave and $\eta $ has to be determined. Substituting the expression for $%
u\left( r,t\right) $ in (\ref{CLaw}), we deduce that $c$ is an eigenvalue
(with $\eta ^{^{\prime }}$ as eigenvector)\ for%
\begin{equation*}
\left( -c+a\left( \eta \left( r-ct\right) \right) \right) \eta ^{^{\prime
}}\left( r-ct\right) =0.
\end{equation*}%
Under appropriate structural conditions on $a\left( \cdot \right) $ (see,
e.g., Section 3), this eigenvalue problem is strictly hyperbolic (cf., e.g.,
\cite[Chapter 11]{Evans}) with speed $c=c\left( \eta \right) >0,$ hence a
wave-like solution $\eta =\eta \left( r-ct\right) $ to (\ref{CLaw}) can be
found. This is a unidirectional wave which travels \emph{into} the region $%
\Omega $. We can now map back to our original coordinate system to find that
$u\left( r,t\right) =\eta \left( x-ct\mathbf{n}\right) $ is a solution to (%
\ref{DYN0}). In plain physical terms, the mass of fluid is carried by the
wave $\eta $ into an infinitesimal layer near the boundary $\Gamma $. This
wave will cease to exist after some small time since once inside $\Omega $,
the primary process is governed by nonlinear diffusion in the porous medium
equation (\ref{gen_PM2}).

It is easy to extend our derivation to systems of the form (\ref{s1})-(\ref%
{s4}) and to give them phyiscal interpretations. For instance, these systems
also occur in the pharmaceutical industry by mathematical modells for the
development of blood coagulation treatments with specific coagulation
factors \cite{EPS, HH, MPE}. The systems (\ref{s1})-(\ref{s4}) are also
motivated by diffusion processes on metric graphs and ramified spaces, which
yield interface problems for quantum graphs with coupled dynamic boundary
conditions at the nodes (see, e.g., \cite{M} and references therein). On the
other hand, the reaction-diffusion equations (\ref{s1}) (for $A_{i}\left(
u_{i}\right) =d_{i}u_{i},$ $i\in \mathbb{N}_{m}$) arise as models for the
densities $u_{i},$ $i\in \mathbb{N}_{m}$ of substances or organisms that
disperse through space by Brownian motion, random walks, hydronamic
turbulance or similar mechanisms. These equations are widely used as models
to account for spatial effects in ecological enviroments \cite{Co}. For
equations (\ref{s1}), the Dirichlet boundary condition (\ref{D}) specifies
the density $u_{i}$ of species at the boundary $\Gamma $, with an
interpretation that anything that reaches the boundary $\Gamma $ of $\Omega $
leaves and does not return. If $u_{0}\equiv 0,$ then (\ref{D}) may be
interpreted as if the species suffers extiction if say the patch $\Gamma _{1}
$ where the individuals live is toxic. The (homogeneous, $\mathbf{H}\equiv
\mathbf{0}$) Neumann boundary condition (\ref{N}) says that nothing can
cross the boundary of $\Omega .$ Another relation in ecological models is
the Robin boundary condition (\ref{R}) with $j\left( x\right) \equiv 0$ and $%
C=C_{i}\in \mathbb{R}^{\ast },$ $i\in \mathbb{N}_{m}$ which can be
interpreted as saying that when organisms reach the boundary some leave it
but some do \emph{not }depending on the sign of $C_{i}$. Finally, the other
\emph{not} so common condition is the Wentzell-type (dynamic) boundary
condition (\ref{PM2bis}) which states that change in the density of
individuals at $\Gamma _{1}$ is a function of their flux in the normal
direction across $\Gamma _{1}$ and some other function of density if no
dispersive effects along $\Gamma _{1}$ are taken into account. Following our
reasonning above, this type of boundary condition (\ref{s3}) can be
interpreted as saying that some individuals may choose to live on the patch $%
\Gamma _{1}$ but some may\emph{\ not} and can choose to return to the region
$\Omega $, where spatial diffusion coupled with reaction in the bulk $\Omega
$\ is the main mechanism for population movement and interaction. Suppose,
for instance, that certain critical resources for a specific population $%
u_{i},$ $i\in \mathbb{N}_{m},$ are available only on $\Gamma _{1}$. Then $%
u_{i}$ must obey the rule%
\begin{equation}
\partial _{t}u_{i}+d_{i}\nabla u_{i}\cdot \mathbf{n}+h\left( x\right)
u_{i}=0,\text{ on }\Gamma _{1}\times \left( 0,\infty \right) ,  \label{CD}
\end{equation}%
which says that the density $u_{i}$ diffuses (in an infinitesimal layer near
$\Gamma _{1}$) toward the patch $\Gamma _{1}$ in the direction of normal
flux. Again the main mechanism for this behavior here is the influence of
external forces on $\Gamma $ on a particular population $u_{i}$. Of course,
in this context the function $h\left( x\right) $ plays the role of a
resource density function on $\Gamma _{1}$, and it can generally depend also
on time. In fact, it is not hard to imagine a typical scenario where
predatory individuals are preferentially concentrated around valued
resources on $\Gamma _{1}$ where the likelihood of prey is greatest. Hence,
in the more general case of (\ref{PM2bis}) the state densities $u_{i}$ may
be also allowed to carry mass on $\Gamma _{1}$ in contrast to the usual
Robin condition for which the mass is always zero. This general description (%
\ref{s3}) along the patch $\Gamma _{1}$ can have substantial consequences on
the dynamics of various ecological enviroments modelled by
reaction-diffusion systems. We give a short reasonning for this behavior as
follows. In the case of a scalar (non-degenerate) diffusion equation ($m=1,$
$a_{1}\left( \cdot \right) \equiv d_{1}$), we have shown in \cite{GalNS, G0}
(say, in dimension $N\geq 3$) that problem (\ref{1.11bb})-(\ref{1.12bb})
posseses a finite dimensional global attractor $\mathcal{A}_{\text{dyn}}$
whose dimension is essentially of\emph{\ different order} than the dimension
of the global attractor $\mathcal{A}_{\text{D-N-R}}$ for the same parabolic
problem (\ref{1.11bb}) with a Dirichlet/Neumann-Robin boundary condition (%
\ref{D})-(\ref{R}) (cf., also \cite{G0}). In particular, the correct
asymptotics for the Hausdorff and, respectively, the fractal dimensions of $%
\mathcal{A}_{\text{dyn}}$\ are%
\begin{equation}
\dim _{H}\mathcal{A}_{\text{dyn}}\sim C\left( f_{1},g_{1}\right) \frac{%
\left\vert \Gamma \right\vert }{\left( \nu b\right) ^{N-1}}\text{, }\dim _{F}%
\mathcal{A}_{\text{dyn}}\sim C\left( f_{1},g_{1}\right) \frac{\left\vert
\Gamma \right\vert }{\left( \nu b\right) ^{N-1}},  \label{dim}
\end{equation}%
as long as $\nu \rightarrow 0^{+}$. Here, $C=C\left( f_{1},g_{1}\right) $ is
a positive constant that is independent of the size of $\Omega ,$ but
depends only on $f_{1}$ and $g_{1}$, and $\left\vert \Gamma \right\vert $
denotes the natural Lebesgue surface measure of $\Gamma \subset \mathbb{R}%
^{2}$. Note that the asymptotics for the dimension of $\mathcal{A}_{\text{%
D-N-R}}$ is actually $C\left( f_{1}\right) \left\vert \Omega \right\vert
/\left( \nu b\right) ^{N/2}$, as $\nu \rightarrow 0^{+}$ (see, e.g., \cite%
{BV}) suggesting that the dynamics on $\mathcal{A}_{\text{dyn}}$ is \emph{%
qualitatively} different than that on $\mathcal{A}_{\text{D-N-R}}$ even
though both systems are gradient like \cite{G0}\ (i.e., both problems
possess a global Lyapunov function). The asymptotic estimates in (\ref{dim})
are essentially determined by the instability indices of a properly chosen
family of (hyperbolic) equilibria $u_{\ast }$ (see, e.g., \cite{BV}, \cite%
{GalNS}). One achieves a lower bound like (\ref{dim}) by computing the
dimension of the unstable eigenspace $E^{u}$ of the linearization of (\ref%
{1.11bb})-(\ref{1.12bb}) around a constant equilibrium $u_{\ast }$. In this
case, the linearized system possesses at least $n\sim C\left(
f_{1},g_{1}\right) \left\vert \Gamma \right\vert /\left( \nu b\right) ^{N-1}$
(as $\nu \rightarrow 0^{+}$) unstable solutions. This points out once again
to the destabilizing nature of the dynamic boundary condition (\ref{CD})
even when the dynamics in the bulk $\Omega $\ is essentially strictly \emph{%
linear} (see, Appendix). We emphasize that this kind of behavior \emph{cannot%
} hold for the Dirichlet/Neumann-Robin boundary condition (\ref{D})-(\ref{R}%
).

\section{Main results}

\label{fs}

The natural phase-space for problems of the form (\ref{s1})-(\ref{s4}) is%
\begin{equation*}
\mathbb{X}^{s_{1},s_{2}}:=L^{s_{1}}(\Omega )\oplus L^{s_{2}}(\Gamma )=\{U=%
\binom{u_{1}}{u_{2}}:\;u_{1}\in L^{s_{1}}(\Omega ),\;u_{2}\in
L^{s_{2}}(\Gamma )\},
\end{equation*}%
$s_{1},s_{2}\in \left[ 1,+\infty \right] ,$ endowed with norm%
\begin{equation}
\left\Vert U\right\Vert _{\mathbb{X}^{s_{1},s_{2}}}=\left( \int_{\Omega
}\left\vert u_{1}\left( x\right) \right\vert ^{s_{1}}dx\right)
^{1/s_{1}}+\left( \int_{\Gamma }\left\vert u_{2}(x)\right\vert
^{s_{2}}dS_{x}\right) ^{1/s_{2}},  \label{2.1}
\end{equation}%
if $s_{1},s_{2}\in \lbrack 1,\infty ),$ and
\begin{align*}
\Vert U\Vert _{\mathbb{X}^{\infty }}& :=\max \{\Vert u_{1}\Vert _{L^{\infty
}(\Omega )},\Vert u_{2}\Vert _{L^{\infty }(\Gamma )}\} \\
& \simeq \Vert u_{1}\Vert _{L^{\infty }(\Omega )}+\Vert u_{2}\Vert
_{L^{\infty }(\Gamma )}.
\end{align*}%
We agree to denote by $\mathbb{X}^{s}$ the space $\mathbb{X}^{s,s}.$
Moreover, we have%
\begin{equation}
\mathbb{X}^{s}=L^{s}\left( \overline{\Omega },d\mu \right) ,\text{ }s\in %
\left[ 1,+\infty \right] ,  \label{2.1b}
\end{equation}%
where the measure $d\mu =dx_{\mid \Omega }\oplus dS_{x}{}_{\mid \Gamma }$ on
$\overline{\Omega }$ is defined for any measurable set $A\subset \overline{%
\Omega }$ by%
\begin{equation}
\mu (A)=|A\cap \Omega |+S(A\cap \Gamma ).  \label{2.1t}
\end{equation}%
Identifying each function $\theta \in C\left( \overline{\Omega }\right) $
with the vector $\Theta =\binom{\theta _{\mid \Omega }}{\theta _{\mid \Gamma
}}$, we have that $C(\overline{\Omega })$ is a dense subspace of $\mathbb{X}%
^{s}$ for every $s\in \lbrack 1,\infty )$ and a closed subspace of $\mathbb{X%
}^{\infty }$. In general, any vector $\theta \in \mathbb{X}^{s}$ will be of
the form $\binom{\theta _{1}}{\theta _{2}}$ with $\theta _{1}\in L^{s}\left(
\Omega ,dx\right) $ and $\theta _{2}\in L^{s}\left( \Gamma ,dS\right) ,$ and
there need not be any connection between $\theta _{1}$ and $\theta _{2}$.

Next, we set%
\begin{equation*}
\mathcal{X}^{s_{1},s_{2}}:=\prod\nolimits_{i\in I_{m}}L^{s_{1}}\left( \Omega
\right) \times \prod\nolimits_{i\in J_{m}}\mathbb{X}^{s_{1},s_{2}},\text{
for any }s_{1},s_{2}\in \left[ 1,+\infty \right] ,
\end{equation*}%
where, for any given set $X,$%
\begin{equation*}
\prod\nolimits_{i\in I}X:=\underset{\left\vert I\right\vert \text{ times}}{%
\underbrace{X\times ...\times X}}.
\end{equation*}%
The norm in the space $\mathcal{X}^{s_{1},s_{2}},$ for any $s_{1},s_{2}\in
\lbrack 1,+\infty )$ is%
\begin{equation}
\left\Vert \overrightarrow{u}\right\Vert _{\mathcal{X}^{s_{1},s_{2}}}:=\sum%
\nolimits_{i\in I_{m}}\left\Vert u_{i}\right\Vert _{L^{s_{1}}\left( \Omega
\right) }+\sum\nolimits_{i\in J_{m}}\left( \left\Vert u_{i}\right\Vert
_{L^{s_{1}}\left( \Omega \right) }+\delta _{i}\left\Vert u_{i}\right\Vert
_{L^{s_{2}}\left( \Gamma \right) }\right) ,  \label{norm}
\end{equation}%
where $\overrightarrow{u}=\left( u_{1},...,u_{m}\right) $, while the norm in
the space $\mathcal{X}^{\infty }:=\mathcal{X}^{\infty ,\infty }$ is
naturally given by%
\begin{equation*}
\left\Vert \overrightarrow{u}\right\Vert _{\mathcal{X}^{\infty }}:=\max
\{\max_{i\in \mathbb{N}_{m}}\left\Vert u_{i}\right\Vert _{L^{\infty }\left(
\Omega \right) },\max_{i\in J_{m}}\left\Vert u_{i}\right\Vert _{L^{\infty
}\left( \Gamma \right) }\}.
\end{equation*}%
If $s_{1}=s_{2}=s,$ we will simply write $\mathcal{X}^{s}$ instead of $%
\mathcal{X}^{s_{1},s_{2}}$. Finally, without further abuse of notation, we
will also refer to $\mathcal{X}^{\overrightarrow{r}},$ $\overrightarrow{r}%
=\left( r_{1},...,r_{m}\right) $, as the following Banach space%
\begin{equation*}
\mathcal{X}^{\overrightarrow{r}}:=\prod\nolimits_{i\in I_{m}}L^{r_{i}}\left(
\Omega \right) \times \prod\nolimits_{i\in J_{m}}\mathbb{X}^{r_{i},r_{i}},
\end{equation*}%
endowed with the natural norm in (\ref{norm}).

Let us now state our main hypotheses on the source terms $f_{i},g_{i},$ $%
h_{i}$ and nonlinear diffusions $a_{i}$, for each $i\in \mathbb{N}_{m}.$

\noindent \textbf{Conditions on} $a_{i}:$ The Carath\'{e}odory functions $%
a_{i}$ (with values in $\mathbb{R}$) satisfy the condition: $\exists \alpha
_{i}>0,$ $\forall s_{i}\in \mathbb{R}$ such that%
\begin{equation}
a_{i}\left( s_{i}\right) \geq \alpha _{i}\left\vert s_{i}\right\vert
^{p_{i}},\text{ }i\in \mathbb{N}_{m},  \label{A1}
\end{equation}%
for some nonnegative $p_{i}.$

\noindent \textbf{Conditions on} $f_{i}$, $g_{i},$ $h_{i}:$ The Carath\'{e}%
odory functions $f_{i}$, $g_{i},$ $h_{i}$ (with values in $\mathbb{R}$)
satisfy the conditions: $\exists C_{f_{i}},C_{g_{i}},C_{h_{i}}>0,$ for
almost all $\left( x,t\right) ,$ $\forall s_{i}\in \mathbb{R}$, such that%
\begin{equation}
\left\{
\begin{array}{l}
\sum\nolimits_{i\in \mathbb{N}_{m}}f_{i}\left( x,t,s_{1},...,s_{m}\right)
s_{i}\geq -\sum\nolimits_{i\in \mathbb{N}_{m}}C_{f_{i}}\left\vert
s_{i}\right\vert ^{2}-\widetilde{C}_{f}, \\
\sum\nolimits_{i\in J_{m}}g_{i}\left( x,t,s_{1},...,s_{m}\right) s_{i}\geq
-\sum\nolimits_{i\in J_{m}}C_{g_{i}}\left\vert s_{i}\right\vert ^{2}-%
\widetilde{C}_{g}, \\
\sum\nolimits_{i\in I_{m}}h_{i}\left( x,t,s_{1},...,s_{m}\right) s_{i}\geq
-\sum\nolimits_{i\in I_{m}}C_{h_{i}}\left\vert s_{i}\right\vert ^{2}-%
\widetilde{C}_{h},%
\end{array}%
\right.  \label{A2}
\end{equation}%
for some nonnegative $\widetilde{C}_{f},$ $\widetilde{C}_{g}$ and $%
\widetilde{C}_{h}$.

The question of global existence and the $L^{p}$-$L^{\infty }$ smoothing
property for solutions of (\ref{s1})-(\ref{s4}) can be stated for the
function $\overrightarrow{u}=\left( u_{1},...,u_{m}\right) $, as follows. We
say that the parabolic system (\ref{s1})-(\ref{s4}) satisfies \textbf{%
Property P}$\left( r_{1},r_{2}\right) $, for some finite $r_{1},r_{2}\geq 1$%
, if, for all $i\in \mathbb{N}_{m}$, any of the following conditions holds:

(i) There exists a positive function $Q,$ independent of initial data, and a
positive constant $\eta $, such that%
\begin{equation}
\sup_{t\geq \eta >0}\left\Vert \overrightarrow{u}\left( t\right) \right\Vert
_{\mathcal{X}^{r_{1},r_{2}}}\leq Q\left( \eta \right) .  \label{dis}
\end{equation}

(ii) There exists a positive constant $\mathcal{C}$, independent of initial
data, such that%
\begin{equation}
\underset{t\rightarrow \infty }{\lim \sup }\left\Vert \overrightarrow{u}%
\left( t\right) \right\Vert _{\mathcal{X}^{r_{1},r_{2}}}\leq \mathcal{C}.
\label{point}
\end{equation}

(iii) If $\overrightarrow{u}_{0}=\left( u_{10},...,u_{m0}\right) \in $ $%
\mathcal{X}^{\infty },$ $i\in \mathbb{N}_{m}$, there exists a positive
function $Q$, independent of initial data, such that%
\begin{equation}
\sup_{t\geq 0}\left\Vert \overrightarrow{u}\left( t\right) \right\Vert _{%
\mathcal{X}^{r_{1},r_{2}}}\leq Q\left( \left\Vert \overrightarrow{u}%
_{0}\right\Vert _{\mathcal{X}^{\infty }}\right) .  \label{global}
\end{equation}

The first result concerning the $L^{\infty }$-estimate for solutions of (\ref%
{s1})-(\ref{s4}) in the non-degenerate case, shows that if either one of the
properties for\textbf{\ P}$\left( s_{1},s_{2}\right) $ above holds apriori
for some finite $s_{1},s_{2}\geq 1$, then it also holds for $%
s_{1}=s_{2}=\infty $.

\begin{theorem}
\label{linf}Let the assumptions (\ref{A1}), (\ref{A2}) be satisfied such
that $p_{i}=0$, for all $i\in \mathbb{N}_{m}$ (i.e., (\ref{non}) holds).
Suppose that the system (\ref{s1})-(\ref{s4}) satisfies the property \textbf{%
P}$\left( 1,1\right) $-(i) (respectively, (ii) or (iii)), then it also
satisfies \textbf{P}$\left( \infty ,\infty \right) $-(i) (respectively, (ii)
or (iii)).
\end{theorem}

\begin{remark}
\label{rem1}Note that we can also consider the more general case in which $%
a_{i}\left( u_{i}\right) $ is replaced by $a_{i}\left( x,t,\overrightarrow{u}%
\right) ,$ $i\in \mathbb{N}_{m}$, that is, the equations (\ref{s1}) are
strongly coupled in their diffusions. In this case, we must replace
assumption (\ref{A1}) by%
\begin{equation}
a_{i}\left( x,t,\overrightarrow{s}\right) \geq \alpha _{i}\left\vert
\overrightarrow{s}\right\vert ^{p_{i}},\text{ }\forall \overrightarrow{s}\in
\mathbb{R}^{m},  \label{remp}
\end{equation}%
for almost all $\left( x,t\right) $, and notice that all the computations
performed in the proof of Theorem \ref{linf} hold automatically since $%
\left\vert \overrightarrow{s}\right\vert ^{p_{i}}\geq \left\vert
s_{i}\right\vert ^{p_{i}}$ for any $\overrightarrow{s}=\left(
s_{1},...,s_{m}\right) \in \mathbb{R}^{m}$ (of course, in that case $p_{i}=0$
by assumption, for all $i\in \mathbb{N}_{m}$). Theorem \ref{linf} can be
also extended to systems with $p$-Laplacian diffusions, i.e.,%
\begin{equation*}
a_{i}=a_{i}\left( u_{i},\left\vert \nabla u_{i}\right\vert ^{\varrho
_{i}-2}\right) \geq \alpha _{i}\left\vert \nabla u_{i}\right\vert ^{\varrho
_{i}-2},
\end{equation*}%
for some $\varrho _{i}\geq 2$, $i\in \mathbb{N}_{m}$, by following, for
instance, \cite{G0}.
\end{remark}

The second result is concerned with the full degenerate case (\ref{s1})-(\ref%
{s4}) when the assumptions\ of Theorem \ref{linf} do not hold (in
particular, if it happens that $p_{i}>0$ for some $i\in \mathbb{N}_{m}$).
The proof is based on a truncation technique which was originally developed
by DeGiorgi to study the regularity of solutions to elliptic equations, and
then extensively used by many authors to study weak solutions to degenerate
parabolic systems, subject to the usual static boundary conditions (see,
e.g., \cite{DiB, LSU}). Here, we extend DeGiorgi's method to problems of the
form (\ref{s1})-(\ref{s4}). In order to avoid additional technicalities due
to the different conditions that one can assign on the boundary $\Gamma $\
for each $u_{i},$ $i\in \mathbb{N}_{m}$, we shall focus our attention to the
case $J_{m}=\mathbb{N}_{m}$ only (i.e., we will assume that $%
I_{m}=\varnothing $). In this case, we require that the following growth
assumptions hold:%
\begin{equation}
\left\{
\begin{array}{l}
\left\vert f_{i}\left( x,t,s_{1},...,s_{m}\right) \right\vert \leq
C_{f}\left( \sum\nolimits_{i\in \mathbb{N}_{m}}\left\vert s_{i}\right\vert
^{\theta _{i}}+1\right) , \\
\left\vert g_{i}\left( x,t,s_{1},...,s_{m}\right) \right\vert \leq
C_{g}\left( \sum\nolimits_{i\in \mathbb{N}_{m}}\left\vert s_{i}\right\vert
^{\beta _{i}}+1\right) ,%
\end{array}%
\right.  \label{ga}
\end{equation}%
for some $\theta _{i},\beta _{i}>0$ and some positive constants $%
C_{f},C_{g}. $

\begin{theorem}
\label{linf3}Let (\ref{ga}) hold, and assume that $\exists \alpha
_{i},\sigma _{i}>0$ such that%
\begin{equation}
\alpha _{i}\left\vert \overrightarrow{s}\right\vert ^{p_{i}}\leq a_{i}\left(
x,t,\overrightarrow{s}\right) \leq \sigma _{i}\left\vert \overrightarrow{s}%
\right\vert ^{p_{i}},\text{ }i\in \mathbb{N}_{m},  \label{A1bis}
\end{equation}%
for any $\overrightarrow{s}=\left( s_{1},...,s_{m}\right) \in \mathbb{R}^{m}$%
. Let%
\begin{equation}
\delta :=\max_{i\in \mathbb{N}_{m}}\left\{ 2,\theta _{i}+1,\frac{p_{i}}{2}%
+1\right\} ,\text{ }\gamma :=\max_{i\in \mathbb{N}_{m}}\left\{ 2,\beta
_{i}+1,\frac{p_{i}}{2}+1\right\} .  \label{exp}
\end{equation}%
Suppose that the system (\ref{s1})-(\ref{s4}) satisfies the property \textbf{%
P}$\left( \delta ,\gamma \right) $-(i) (respectively, (ii) or (iii)), then
it also satisfies \textbf{P}$\left( \infty ,\infty \right) $-(i)
(respectively, (ii) or (iii)). In particular, for every $i\in \mathbb{N}_{m}$
and $T,\tau >0$ such that $T-2\tau >0,$ the following estimate holds:%
\begin{align}
& \sup_{\left( x,t\right) \in \left[ T-\tau ,T\right] \times \overline{%
\Omega }}\left\vert u_{i}\left( x,t\right) \right\vert  \label{llinf} \\
& \leq Q\left( 1+\left\Vert \overrightarrow{u}\right\Vert _{L^{\delta
}\left( \left[ T-2\tau ,T\right] \times \Omega \right) }+\left\Vert
\overrightarrow{u}\right\Vert _{L^{\gamma }\left( \left[ T-2\tau ,T\right]
\times \Gamma \right) }\right) ,  \notag
\end{align}%
for some positive function $Q$ which is independent of $\overrightarrow{u},$
time and the initial data. The function $Q$ can be computed explicitly in
terms of the physical paramaters of the problem.
\end{theorem}

We will now show how to deduce the property \textbf{P}$\left(
s_{1},s_{2}\right) ,$ for some finite $s_{1},s_{2}\geq 1,$ for the problem (%
\ref{s1})-(\ref{s2}) subject to a dynamic boundary condition of the form (%
\ref{s4}). We shall first consider a special case. Let $\Gamma \subset
\mathbb{R}^{N-1}$\ consists of two disjoint open subsets $\Gamma _{1}$\ and $%
\Gamma _{2}$, each $\overline{\Gamma }_{i}\backslash \Gamma _{i}$\ is a $S$%
-null subset of $\Gamma $\ and $\Gamma =\overline{\Gamma }_{1}\cup \overline{%
\Gamma }_{2}$, such that $u_{i},$ \thinspace $i\in \mathbb{N}_{m}$ satisfy (%
\ref{s3}) on $\Gamma _{1}\times \left( 0,\infty \right) ,$ and%
\begin{equation}
u_{i}=0\text{, on }\Gamma _{2}\times \left( 0,\infty \right) ,\text{ for }%
i\in \mathbb{N}_{m}.  \label{s5}
\end{equation}%
We assume that $\Gamma _{2}$ is a set of positive surface measure, and that
the nonlinearities $f_{i},g_{i}$ satisfy the following special form of (\ref%
{A2}), that is,%
\begin{equation}
\left\{
\begin{array}{l}
\sum\nolimits_{i\in \mathbb{N}_{m}}f_{i}\left( x,t,s_{1},...,s_{m}\right)
s_{i}\left\vert s_{i}\right\vert ^{m_{i}-2}\geq -\sum\nolimits_{i\in \mathbb{%
N}_{m}}C_{f_{i}}\left\vert s_{i}\right\vert ^{m_{i}}-\widetilde{C}_{f}, \\
\sum\nolimits_{i\in \mathbb{N}_{m}}g_{i}\left( x,t,s_{1},...,s_{m}\right)
s_{i}\left\vert s_{i}\right\vert ^{m_{i}-2}\geq -\sum\nolimits_{i\in \mathbb{%
N}_{m}}C_{g_{i}}\left\vert s_{i}\right\vert ^{m_{i}}-\widetilde{C}_{g},%
\end{array}%
\right.  \label{A3}
\end{equation}%
for some $m_{i}\geq 1,$ and some positive constants $C_{f_{i}},C_{g_{i}}$
and $\widetilde{C}_{f},\widetilde{C}_{g}\geq 0.$ Note that (\ref{A2}) is
equivalent to (\ref{A3}) for $m_{i}=2.$

The second main result gives a $\mathcal{X}^{\overrightarrow{r}}$%
-dissipative estimate for solutions of (\ref{s1}), (\ref{s3}), (\ref{s4}), (%
\ref{s5}).

\begin{theorem}
\label{diss}Suppose $\Gamma _{2}$ is a set of positive surface measure. Let
the assumptions (\ref{A1}), (\ref{A3}) be satisfied, and let $p_{i}>0$ for
all $i\in \mathbb{N}_{m}$. Then, the system (\ref{s1}), (\ref{s3}), (\ref{s4}%
), (\ref{s5}) satisfies property \textbf{P}$\left( \overrightarrow{r}\right)
$-(i) for $\overrightarrow{r}=\left( r_{1},...,r_{m}\right) $ with $%
r_{i}=m_{i},$ $i\in \mathbb{N}_{m}.$ Moreover, if $\overrightarrow{u}_{0}\in
\mathcal{X}^{\overrightarrow{r}},$ then there also exists a positive
function $Q$, independent of initial data and time, such that%
\begin{equation*}
\sup_{t\geq 0}\left\Vert \overrightarrow{u}\left( t\right) \right\Vert _{%
\mathcal{X}^{\overrightarrow{r}}}\leq Q\left( \left\Vert \overrightarrow{u}%
_{0}\right\Vert _{\mathcal{X}^{\overrightarrow{r}}}\right) .
\end{equation*}
\end{theorem}

\begin{remark}
Theorem \ref{diss} \emph{only} holds if $\Gamma \neq \Gamma _{1}$, i.e.,
when the boundary $\Gamma _{2}$ has positive measure (cf. also Remark \ref%
{remi} below). We require different arguments for the case when $\Gamma
_{2}\equiv \varnothing $.
\end{remark}

We shall now derive another dissipative estimate for solutions of (\ref{s1}%
)-(\ref{s4}) in $\mathcal{X}^{\overrightarrow{r}}$-norm which also covers
Dirichlet boundary conditions (for $i\in I_{m}$) and applies to the case
when $\Gamma _{2}\equiv \varnothing $, without enforcing any further sign
restrictions on all the $p_{i}$'s (compare with the assumptions in Theorem %
\ref{diss}). Analogous to (\ref{A3}), we shall assume that the functions $%
f_{i},g_{i}$ satisfy%
\begin{equation}
\left\{
\begin{array}{l}
\sum\nolimits_{i\in \mathbb{N}_{m}}f_{i}\left( x,t,s_{1},...,s_{m},\zeta
_{i}\right) s_{i}\left\vert s_{i}\right\vert ^{m_{i}-2}\geq
-\sum\nolimits_{i\in \mathbb{N}_{m}}C_{f_{i}}\left\vert s_{i}\right\vert
^{m_{i}+p_{i}}-\widetilde{C}_{f}, \\
\sum\nolimits_{i\in \mathbb{N}_{m}}g_{i}\left( x,t,s_{1},...,s_{m}\right)
s_{i}\left\vert s_{i}\right\vert ^{m_{i}-2}\geq -\sum\nolimits_{i\in \mathbb{%
N}_{m}}C_{g_{i}}\left\vert s_{i}\right\vert ^{m_{i}+p_{i}}-\widetilde{C}_{g},%
\end{array}%
\right.  \label{A3bis}
\end{equation}%
for some $m_{i}>1,$ and some \emph{real} constants $C_{f_{i}},C_{g_{i}}$ and
$\widetilde{C}_{f},\widetilde{C}_{g}\geq 0.$ Moreover, consider the
(self-adjoint) eigenvalue problem for so-called Wentzell Laplacians $\Delta
_{W,i}$\ (see \cite[Appendix]{GalNS}), as follows:%
\begin{equation}
-a_{i}\Delta \varphi _{i}-C_{f_{i}}\varphi _{i}=\Lambda _{i}\varphi \text{
in }\Omega ,  \label{wi}
\end{equation}%
where%
\begin{equation*}
a_{i}:=\alpha _{i}\left( m_{i}-1\right) \left( \frac{2}{m_{i}+p_{i}}\right)
^{2},
\end{equation*}%
with a boundary condition that depends on the eigenvalue $\Lambda _{i}$
explicitly,%
\begin{equation}
a_{i}\partial _{\mathbf{n}}\varphi _{i}-C_{g_{i}}\varphi _{i}=\Lambda
_{i}\varphi _{i}\text{ on }\Gamma _{1},  \label{wib1}
\end{equation}%
such that%
\begin{equation}
\varphi _{i}=0\text{ on }\Gamma _{2}.  \label{wib2}
\end{equation}%
Here $\Gamma _{2}$ is assumed to be a set of nonnegative surface measure
(the case $\Gamma _{2}=\varnothing $ may be also allowed). Our second
condition on the nonlinearities $f_{i},$ $g_{i}$ is concerned with some sign
assumptions on $C_{f_{i}}$ and $C_{g_{i}}$. In particular, we assume that%
\begin{equation}
\Lambda _{1}:=\inf_{i\in \mathbb{N}_{m}}\Lambda _{1,i}>0.  \label{eigen}
\end{equation}%
Observe that, if $C_{f_{i}}<0,$ $C_{g_{i}}<0,$ for all $i\in \mathbb{N}_{m},$
then the first eigenvalue $\Lambda _{1,i}$ of (\ref{wi})-(\ref{wib2}) is
always positive, for any $i\in \mathbb{N}_{m}.$ Otherwise, we note that even
when at least one of $C_{f_{i}}$ or $C_{g_{i}}$ is positive, it may still
happen that (\ref{eigen}) holds. In this sense, our system (\ref{s1})-(\ref%
{s4}) becomes dissipative even when at least one of the terms $f_{i},g_{i}$
has the wrong sign at infinity, but the bad sign is compensated by the other
term (see, also \cite{Ma} for similar assumptions).

\begin{theorem}
\label{diss2}Suppose that $\Gamma _{2}$ is a set of nonnegative surface
measure. Let (\ref{A1}), (\ref{A3bis}), (\ref{eigen}) hold, and assume that $%
p_{i}>0$ for all $i\in \mathbb{N}_{m}$. Then, the system (\ref{s1}), (\ref%
{s4}), (\ref{s5}) satisfies property \textbf{P}$\left( \overrightarrow{r}%
\right) $-(i), for $\overrightarrow{r}=\left( r_{1},...,r_{m}\right) $ with $%
r_{i}=m_{i},$ $i\in \mathbb{N}_{m}.$ On the other hand, if $p_{i}=0$ for
some $i\in \mathbb{N}_{m}$, and if $\overrightarrow{u}_{0}\in \mathcal{X}^{%
\overrightarrow{r}},$ then this system satisfies \textbf{P}$\left(
\overrightarrow{r}\right) $-(ii) instead.
\end{theorem}

\begin{remark}
All the above results can be extended to models which also incorporate
advection effects in the domain $\Omega $ and on the boundary $\Gamma $
(i.e., the reaction terms $f_{i},g_{i}$ and $h_{i}$ may also depend on $%
\nabla u_{i}$ and $\nabla _{\Gamma }u_{i}$, respectively). Indeed, by making
similar assumptions to (\ref{A2})-(\ref{A3bis}), any integral over $%
\left\vert \nabla u_{i}\right\vert $ may be absorbed by diffusion in the
bulk $\Omega $ with an appropiate application of suitable Young and Sobolev
inequalities (cf. Section 4). We will return to these questions elsewhere.
\end{remark}

\section{Proof of main results}

\subsection{Proof of Theorem \protect\ref{linf}}

In order to justify our computations, we need to construct an approximation
scheme which relies on the existence of classical (smooth)\ solutions to the
non-degenerate analogue of (\ref{s1})-(\ref{s4}) (if at least one $p_{i}\neq
0$). One of the advantages of this construction is that now every (possibly,
very weak)\ solution can be approximated by regular ones and the
justification of our estimates for such solutions is immediate. To this end,
for each $\epsilon >0$, let us consider the following non-degenerate
parabolic system:%
\begin{equation}
\partial _{t}u_{i}-\text{div}\left( a_{i}^{\epsilon }\left( u_{i}\right)
\nabla u_{i}\right) +f_{i}\left( x,t,\overrightarrow{u}\right) =0\text{, in }%
\Omega \times \left( 0,\infty \right) ,  \label{s1c}
\end{equation}%
for $i=1,...,m,$ where $a_{i}^{\epsilon }\left( s_{i}\right) :=a_{i}\left(
s_{i}+\epsilon \right) \geq O\left( \epsilon \right) >0,$ subject to the
following set of boundary conditions%
\begin{equation}
\partial _{\mathbf{n}}u_{i}+h_{i}\left( x,t,\overrightarrow{u}\right) =0,%
\text{ on }\Gamma \times \left( 0,\infty \right) ,\text{ \thinspace }i\in
I_{m}  \label{s2c}
\end{equation}%
and%
\begin{equation}
\delta _{i}\partial _{t}u_{i}+a_{\epsilon }\left( u_{i}\right) \partial _{%
\mathbf{n}}\left( u_{i}\right) +g_{i}\left( x,t,\overrightarrow{u}\right) =0,%
\text{ on }\Gamma \times \left( 0,\infty \right) ,\text{ \thinspace }i\in
J_{m}.  \label{s3c}
\end{equation}%
We equip the system (\ref{s1})-(\ref{s3}) with the initial conditions%
\begin{equation}
u_{i\mid t=0}^{\epsilon }=u_{i0}^{\epsilon }\text{ in }\Omega ,\text{ }%
u_{i\mid t=0}^{\epsilon }=u_{i0}^{\epsilon }\text{ on }\Gamma ,  \label{s4c}
\end{equation}%
for $u_{i}^{\epsilon }\left( 0\right) :=u_{i0}^{\epsilon }\in C^{\infty
}\left( \overline{\Omega }\right) ,$ $i\in \mathbb{N}_{m}$, such that%
\begin{equation*}
u_{i}^{\epsilon }\left( 0\right) \rightarrow u_{i0}\text{ in }L^{s}\left(
\Omega \right) \text{, }u_{i}^{\epsilon }\left( 0\right) _{\mid \Gamma
}\rightarrow v_{i0}\text{ in }L^{s}\left( \Gamma \right) ,
\end{equation*}%
for some given $s\geq 1.$ Then, the approximate problem (\ref{s1c})-(\ref%
{s4c}) admits a unique (smooth) classical solution with%
\begin{equation*}
\overrightarrow{u}^{\epsilon }=\left( u_{1}^{\epsilon },...,u_{m}^{\epsilon
}\right) \in C^{1}\left( \left[ 0,t_{\ast }\right] ;\left( C^{\infty }\left(
\overline{\Omega }\right) \right) ^{m}\right) ,
\end{equation*}%
for some $t_{\ast }>0$ and each $\epsilon >0$ (see \cite{Escher, Escher2,
CJ, Ma}). Being pedants, we cannot apply the main results of \cite{Escher,
Ma} (cf. also \cite{Escher2}) directly to equations (\ref{s1c})-(\ref{s4c})
since the functions $a_{i}^{\epsilon },$ $f_{i},$ $g_{i}$ and $h_{i}$ are
not smooth enough. Moreover, the solutions constructed this way may only
exists locally in time for some interval $\left[ 0,t_{\ast }\right] $.
However, by approximating the functions $a_{i}^{\epsilon },$ $f_{i},$ $g_{i}$%
, $h_{i}$ by smooth ones, say, in $C^{\infty }\left( \mathbb{R},\mathbb{R}%
\right) $, we may apply Remark \ref{rem} below for the solutions of the
approximate equations (\ref{s1c})-(\ref{s4c}), and deduce the existence of a
globally well-defined solution on $\left[ 0,T\right] $, for any $T>0$.
Indeed, an apriori global bound in H\"{o}lder-norm for $\overrightarrow{u}%
^{\epsilon }$ guarantees the global existence of classical solutions (see,
e.g., \cite{Ma}). Nevertheless, even when these bounds are not available
apriori, we may still choose to work with locally-defined in $\left[
0,t_{\ast }\right] ,$ smooth solutions $\overrightarrow{u}^{\epsilon }$,
that are globally defined on $\mathbb{R}_{+}$ in the lower-order $L^{p}$%
-norms; this turns out to be sufficient for our purposes. Indeed, if the
solution $\overrightarrow{u}^{\epsilon }$ is globally-defined on $\left[ 0,T%
\right] $ in $\mathcal{X}^{r}$-norm, then it will also be global in $%
\mathcal{X}^{\infty }$-norm by (iii). As we shall see in this section, a
global bound in $\mathcal{X}^{r}$-norm can be established for the same
assumptions (\ref{A1})-(\ref{A3bis}) on the nonlinearities.

We begin with the proof of Theorem \ref{linf}, by following similar
arguments to \cite{Du, Du3} for the system (\ref{s1}) with static boundary
conditions. From now on, $c$ will denote a positive constant that is
independent of $t,$ $\epsilon $, $n,$ $\overrightarrow{u}$ and initial data,
which only depends on the other structural parameters of the problem. Such a
constant may vary even from line to line. Moreover, we shall denote by $%
Q_{\tau }\left( n\right) $ a monotone nondecreasing function in $n$ of order
$\tau ,$ for some nonnegative constant $\tau ,$ independent of $n.$ More
precisely, $Q_{\tau }\left( n\right) \sim cn^{\tau }$ as $n\rightarrow
+\infty .$ We begin by showing that the $\mathcal{X}^{n}$-norm of $%
\overrightarrow{u}=\overrightarrow{u}^{\epsilon }$ satisfies a local
recursive relation which can be used to perform an iterative argument. We
divide the proof of Theorem \ref{linf} into several steps.

\noindent \textbf{Step 1} (The basic energy estimate in $\mathcal{X}^{n+1}$%
). We multiply (\ref{s1c}) by $\left\vert u_{i}\right\vert ^{n-1}u_{i},$ $%
n\geq 1,$ and integrate over $\Omega $, for each $i\in \mathbb{N}_{m}$. We
obtain%
\begin{align}
& \frac{1}{\left( n+1\right) }\frac{d}{dt}\left\Vert u_{i}\right\Vert
_{L^{n+1}\left( \Omega \right) }^{n+1}+\left\langle f_{i}\left( x,t,%
\overrightarrow{u}\right) ,\left\vert u_{i}\right\vert
^{n-1}u_{i}\right\rangle _{L^{2}\left( \Omega \right) }  \label{eqn2} \\
& +n\int_{\Omega }a_{i}^{\epsilon }\left( u_{i}\right) \left\vert \nabla
u_{i}\right\vert ^{2}\left\vert u_{i}\right\vert ^{n-1}dx  \notag \\
& =\int_{\Gamma }a_{i}^{\epsilon }\left( u_{i}\right) \partial _{\mathbf{n}%
}u_{i}\left\vert u_{i}\right\vert ^{n-1}u_{i}dS.  \notag
\end{align}%
Similarly, we multiply (\ref{s2c}) and (\ref{s3c}) by $\left\vert
u_{i}\right\vert ^{n-1}u_{i}$ and integrate each relation over $\Gamma $. We
have%
\begin{align}
& \frac{\delta _{i}}{\left( n+1\right) }\frac{d}{dt}\left\Vert
u_{i}\right\Vert _{L^{n+1}\left( \Gamma \right) }^{n+1}+\int_{\Gamma
}a_{i}^{\epsilon }\left( u_{i}\right) \partial _{\mathbf{n}}u_{i}\left\vert
u_{i}\right\vert ^{n-1}u_{i}dS  \label{eqn2bis} \\
& +\left\langle g_{i}\left( x,t,\overrightarrow{u}\right) ,\left\vert
u_{i}\right\vert ^{n-1}u_{i}\right\rangle _{L^{2}\left( \Gamma \right) }
\notag \\
& =0,  \notag
\end{align}%
for each $i\in J_{m}$, and%
\begin{equation}
\int_{\Gamma }a_{i}^{\epsilon }\left( u_{i}\right) \partial _{\mathbf{n}%
}u_{i}\left\vert u_{i}\right\vert ^{n-1}u_{i}dS+\left\langle a_{i}^{\epsilon
}\left( u_{i}\right) h_{i}\left( x,t,\overrightarrow{u}\right) ,\left\vert
u_{i}\right\vert ^{n-1}u_{i}\right\rangle _{L^{2}\left( \Gamma \right) }=0,%
\text{ }i\in I_{m}.  \label{eqn3tris}
\end{equation}%
In the case when (\ref{s2}) is replaced by a Dirichlet boundary condition
for $u_{i},$ $i\in I_{m},$ equation (\ref{eqn3tris}) still holds since $%
h_{i}\equiv 0$ in that case.

Let us first observe that, in light of assumption (\ref{A1}), we have $%
a_{i}^{\epsilon }\left( s_{i}\right) \geq \alpha _{i},$ $\forall s_{i}\in
\mathbb{R}$, $\epsilon >0$ (recall that $p_{i}=0$), which immediately implies%
\begin{equation}
\int_{\Omega }a_{i}^{\epsilon }\left( u_{i}\right) \left\vert \nabla
u_{i}\right\vert ^{2}\left\vert u_{i}\right\vert ^{n-1}dx\geq \alpha
_{i}\int_{\Omega }\left\vert \nabla u_{i}\right\vert ^{2}\left\vert
u_{i}\right\vert ^{n-1}dx,\text{ }i\in \mathbb{N}_{m}.  \label{eqn3b}
\end{equation}%
Moreover, on account of of the assumptions (\ref{A2}) for $f_{i}$, $g_{i}$
and a basic application of H\"{o}lder and Young inequalities, we deduce%
\begin{equation}
\sum\nolimits_{i\in \mathbb{N}_{m}}\left\langle f_{i}\left( x,t,%
\overrightarrow{u}\right) ,\left\vert u_{i}\right\vert
^{n-1}u_{i}\right\rangle _{L^{2}\left( \Omega \right) }\geq
-c\sum\nolimits_{i\in \mathbb{N}_{m}}\left\Vert u_{i}\right\Vert
_{L^{n+1}\left( \Omega \right) }^{n+1}-c,  \label{ineqf}
\end{equation}%
and%
\begin{equation}
\sum\nolimits_{i\in J_{m}}\left\langle g_{i}\left( x,t,\overrightarrow{u}%
\right) ,\left\vert u_{i}\right\vert ^{n-1}u_{i}\right\rangle _{L^{2}\left(
\Gamma \right) }\geq -c\sum\nolimits_{i\in J_{m}}\delta _{i}\left\Vert
u_{i}\right\Vert _{L^{n+1}\left( \Gamma \right) }^{n+1}-c.  \label{ineqg}
\end{equation}

In order to estimate the source terms involving $h_{i}$ on the boundary (\ref%
{eqn3tris}), we need the following lemma which allows us to control surface
integrals in terms of volume integrals (see, e.g., \cite{G0}).

\begin{lemma}
\label{boundary}Let $n\geq 1$, $p\geq 0,$ $s>-1.$ Then for every $%
\varepsilon >0,$ there holds%
\begin{equation*}
\int_{\Gamma }\left\vert u\right\vert ^{s+n}dS\leq \varepsilon \left(
s+n\right) \int_{\Omega }\left\vert \nabla u\right\vert ^{2}\left\vert
u\right\vert ^{p+n-1}dx+\frac{C}{\varepsilon }\left( s+n\right) \left(
\left\Vert u\right\Vert _{L^{s+n}\left( \Omega \right) }^{s+n}+1\right) ,
\end{equation*}%
for some positive constant $C=C\left( p,s\right) $ independent of $u,$ $%
\varepsilon $ and $n.$
\end{lemma}

Applying Lemma \ref{boundary} to $u=u_{i},$ with $s=1$ and $p=0,$ \thinspace
$i\in I_{m}$, we have%
\begin{align}
\int_{\Gamma }\left\vert u_{i}\right\vert ^{1+n}dS& \leq \varepsilon \left(
n+1\right) \int_{\Omega }\left\vert \nabla u_{i}\right\vert ^{2}\left\vert
u\right\vert ^{n-1}dx  \label{ineqh} \\
& +\frac{C}{\varepsilon }\left( n+1\right) \left( \left\Vert
u_{i}\right\Vert _{L^{n+1}\left( \Omega \right) }^{n+1}+1\right) .  \notag
\end{align}%
Thus, in light of the third assumption of (\ref{A2}), using (\ref{A1}), we
find%
\begin{align}
& \sum\nolimits_{i\in I_{m}}\left\langle a_{i}^{\epsilon }\left(
u_{i}\right) h_{i}\left( x,t,\overrightarrow{u}\right) ,\left\vert
u_{i}\right\vert ^{n-1}u_{i}\right\rangle _{L^{2}\left( \Gamma \right) }
\label{ineqh2} \\
& \geq -\sum\nolimits_{i\in I_{m}}\left( \alpha _{i}C_{h_{i}}\left\Vert
u_{i}\right\Vert _{L^{n+1}\left( \Gamma \right) }^{n+1}+\widetilde{C}%
_{h_{i}}\left\Vert u_{i}\right\Vert _{L^{n-1}\left( \Gamma \right)
}^{n-1}\right)  \notag \\
& \geq -c\sum\nolimits_{i\in I_{m}}\alpha _{i}\left\Vert u_{i}\right\Vert
_{L^{n+1}\left( \Gamma \right) }^{n+1}-c  \notag
\end{align}%
which can be bounded below by%
\begin{align}
& -c\varepsilon \sum\nolimits_{i\in I_{m}}\alpha _{i}\left( n+1\right)
\int_{\Omega }\left\vert \nabla u_{i}\right\vert ^{2}\left\vert
u_{i}\right\vert ^{n-1}dx  \label{ineqh2b} \\
& -\frac{c}{\varepsilon }\sum\nolimits_{i\in I_{m}}\left( n+1\right)
(\left\Vert u_{i}\right\Vert _{L^{n+1}\left( \Omega \right) }^{n+1}+1),
\notag
\end{align}%
exploiting the estimate (\ref{ineqh}). Choose now $\varepsilon =\varepsilon
_{n}>0$ in (\ref{ineqh2b}) such that%
\begin{equation}
\varepsilon _{n}=\max_{i\in \mathbb{N}_{m}}\frac{cn\alpha _{i}}{2\left(
n+1\right) },\text{ }\forall n\geq 1,  \label{epsi}
\end{equation}%
and note that $\varepsilon _{n}\leq c$, uniformly as $n\rightarrow \infty $.
Summing the equations (\ref{eqn2}) (respectively, (\ref{eqn2bis}) and (\ref%
{eqn3tris})) over the sets $i\in \mathbb{N}_{m}$ (respectively, $i\in J_{m}$
and $i\in I_{m}$), then adding the relations that we obtain, on account of (%
\ref{ineqf})-(\ref{ineqg}), (\ref{ineqh2})-(\ref{ineqh2b}) we deduce%
\begin{align}
& \frac{d}{dt}\left\Vert \overrightarrow{u}\right\Vert _{\mathcal{X}%
^{n+1}}^{n+1}+cn\left( n+1\right) \sum\nolimits_{i\in \mathbb{N}%
_{m}}\int_{\Omega }\left\vert \nabla u_{i}\right\vert ^{2}\left\vert
u_{i}\right\vert ^{n-1}dx  \label{ine} \\
& \leq Q_{2}\left( n\right) \left( \left\Vert \overrightarrow{u}\right\Vert
_{\mathcal{X}^{n+1}}^{n+1}+1\right) ,  \notag
\end{align}%
for any $n\geq 1$. Here, the function $Q_{2}\left( n\right) \sim n^{2}$ as $%
n\rightarrow \infty $.

\noindent \textbf{Step 2} (The local relation). Set $n_{k}=2^{k}-1,$ $k\geq
0 $, and define%
\begin{equation}
\mathcal{Y}_{k}\left( t\right) :=\left\Vert \overrightarrow{u}\left(
t\right) \right\Vert _{\mathcal{X}^{n_{k}+1}}^{n_{k}+1},  \label{def}
\end{equation}%
for all $k\geq 0$. Then, using the basic identity for $u=u_{i},$%
\begin{equation}
\int_{\Omega }\left\vert \nabla u\right\vert ^{2}\left\vert u\right\vert
^{n-1}dx=\left( \frac{2}{n+1}\right) ^{2}\int_{\Omega }\left\vert \nabla
\left\vert u\right\vert ^{\frac{n+1}{2}}\right\vert ^{2}dx,  \label{ide}
\end{equation}%
from (\ref{ine}) it holds%
\begin{equation}
\frac{d}{dt}\mathcal{Y}_{k}\left( t\right) +\sum\nolimits_{i\in \mathbb{N}%
_{m}}\gamma _{n_{k}}\int_{\Omega }\left\vert \nabla \left\vert
u_{i}\right\vert ^{\frac{n_{k}+1}{2}}\right\vert ^{2}dx\leq Q_{2}\left(
n_{k}\right) \left( \mathcal{Y}_{k}\left( t\right) +1\right) ,  \label{e11}
\end{equation}%
for all $k\geq 0$, \ where $0<\gamma _{0}\leq \gamma _{n_{k}}\sim c$, as $%
n_{k}\rightarrow \infty $. Let $t,\mu $ be two positive constants such that $%
t-\mu /n_{k}>0$. Their precise values will be chosen later. We claim that%
\begin{equation}
\mathcal{Y}_{k}\left( t\right) \leq M_{k}\left( t,\mu \right) :=c_{0}\left(
\mu \right) \left( n^{k}\right) ^{\sigma }(\sup_{s\geq t-\mu /n_{k}}\mathcal{%
Y}_{k-1}\left( s\right) +1)^{\theta _{k}},\text{ }\forall k\geq 1,
\label{claim2}
\end{equation}%
where $c_{0},$ $\sigma $ are positive constants independent of $k,$ and $%
\theta _{k}\geq 1$ is a bounded sequence for all $k.$ The constant $%
c_{0}\left( \mu \right) $ is bounded if $\mu $ is bounded away from zero.

We will now prove (\ref{claim2}) when $2<N.$ The case $N\leq 2$ requires
only minor adjustments. We will follow an argument similar to the proof of
\cite[Theorem 2.3]{GalNS} (cf. also \cite{G0}). For each $k\geq 0$, we define%
\begin{equation*}
r_{i,k}:=\frac{N\left( n_{k}+1\right) -\left( N-2\right) \left(
1+n_{k}\right) }{N\left( n_{k}+1\right) -\left( N-2\right) \left(
1+n_{k-1}\right) },\text{ }s_{i,k}:=1-r_{i,k}.
\end{equation*}%
We aim to estimate the term on the right-hand side of (\ref{ine}) in terms
of the $\mathcal{X}^{1+n_{k-1}}$-norm of $\overrightarrow{u}.$ First, H\"{o}%
lder and Sobolev inequalities (with the equivalent norm of Sobolev spaces in
$W^{1,2}\left( \Omega \right) \subset L^{p_{s}}\left( \Omega \right) $, $%
p_{s}=2N/\left( N-2\right) $) yield%
\begin{align}
\int_{\Omega }\left\vert u_{i}\right\vert ^{1+n_{k}}dx& \leq \left(
\int_{\Omega }\left\vert u_{i}\right\vert ^{\frac{\left( n_{k}+1\right) N}{%
N-2}}dx\right) ^{s_{i,k}}\left( \int_{\Omega }\left\vert u_{i}\right\vert
^{1+n_{k-1}}dx\right) ^{r_{i,k}}  \label{ee5b} \\
& \leq c\left( \int_{\Omega }\left\vert \nabla \left\vert u_{i}\right\vert ^{%
\frac{\left( n_{k}+1\right) }{2}}\right\vert ^{2}dx+\int_{\Omega }\left\vert
u_{i}\right\vert ^{1+n_{k}}dx\right) ^{\overline{s}_{i,k}}  \notag \\
& \times \left( \int_{\Omega }\left\vert u_{i}\right\vert
^{1+n_{k-1}}dx\right) ^{r_{i,k}},  \notag
\end{align}%
with $\overline{s}_{i,k}:=s_{i,k}N/\left( N-2\right) \in \left( 0,1\right) $%
. Applying Young's inequality on the right-hand side of (\ref{ee5b}), we get%
\begin{equation}
\int_{\Omega }\left\vert u_{i}\right\vert ^{1+n_{k}}dx\leq \frac{\gamma
_{n_{k}}}{4}\int_{\Omega }\left\vert \nabla \left\vert u_{i}\right\vert ^{%
\frac{n_{k}+1}{2}}\right\vert ^{2}dx+Q_{\tau _{1}}\left( n_{k}\right) \left(
\int_{\Omega }\left\vert u_{i}\right\vert ^{1+n_{k-1}}dx\right) ^{z_{i,k}},
\label{bulk}
\end{equation}%
for some positive constant $\tau _{1}$ independent of $n_{k},$ and where%
\begin{equation*}
z_{i,k}:=r_{i,k}/\left( 1-\overline{s}_{i,k}\right) \geq 1
\end{equation*}%
is bounded for all $k$. Note that we can choose $\tau _{2}$ to be some fixed
positive number since $Q_{\tau _{2}}$ also depends on $\gamma _{n_{k}}\sim c$%
.

To treat the boundary terms on the right-hand side of (\ref{ine}), we define
for $k\geq 0,$%
\begin{equation*}
y_{i,k}:=\frac{\left( N-1\right) \left( n_{k}+1\right) -\left( N-2\right)
\left( 1+n_{k}\right) }{\left( N-1\right) \left( n_{k}+1\right) -\left(
N-2\right) \left( 1+n_{k-1}\right) },\text{ }x_{i,k}:=1-y_{i,k}.
\end{equation*}%
On account of H\"{o}lder and Sobolev inequalities (e.g., $W^{1,2}\left(
\Omega \right) \subset L^{q_{s}}\left( \Gamma \right) ,$ $q_{s}=2\left(
N-1\right) /\left( N-2\right) $), we obtain%
\begin{align}
\int_{\Gamma }\left\vert u_{i}\right\vert ^{1+n_{k}}dS& \leq c\left(
\int_{\Gamma }\left\vert u_{i}\right\vert ^{\frac{\left( N-1\right) \left(
n_{k}+1\right) }{N-2}}dS\right) ^{x_{i,k}}\left( \int_{\Gamma }\left\vert
u_{i}\right\vert ^{1+n_{k-1}}dS\right) ^{y_{i,k}}  \label{ee6b} \\
& \leq c\left( \int_{\Omega }\left\vert \nabla \left\vert u_{i}\right\vert ^{%
\frac{\left( n_{k}+1\right) }{2}}\right\vert ^{2}dx+\int_{\Omega }\left\vert
u_{i}\right\vert ^{1+n_{k}}dx\right) ^{\overline{x}_{i,k}}  \notag \\
& \times \left( \int_{\Gamma }\left\vert u_{i}\right\vert
^{1+n_{k-1}}dS\right) ^{y_{i,k}},  \notag
\end{align}%
with $\overline{x}_{i,k}:=x_{i,k}\left( N-1\right) /\left( N-2\right) $.
Since $\overline{x}_{i,k}\in \left( 0,1\right) $, we can apply Young's
inequality on the right-hand side of (\ref{ee6b}), use the estimate for the $%
L^{1+n_{k}}\left( \Omega \right) $-norm of $u_{i}$ from (\ref{ee5b}) in
order to deduce the following estimate:%
\begin{equation}
\int_{\Gamma }\left\vert u_{i}\right\vert ^{1+n_{k}}dS\leq \frac{\gamma
_{n_{k}}}{4}\int_{\Omega }\left\vert \nabla \left\vert u_{i}\right\vert ^{%
\frac{n_{k}+1}{2}}\right\vert ^{2}dx+Q_{\tau _{2}}\left( n_{k}\right) \left(
\int_{\Omega }\left\vert u_{i}\right\vert ^{1+n_{k-1}}dx\right) ^{l_{i,k}},
\label{bd}
\end{equation}%
for some positive constant $\tau _{2}$ depending on $\tau _{1},$ but which
is independent of $n_{k}$, and where%
\begin{equation*}
l_{i,k}:=\frac{y_{i,k}}{\left( 1-\overline{x}_{i,k}\right) }\geq 1
\end{equation*}%
is bounded for all $k\geq 0$. Inserting estimates (\ref{bulk})-(\ref{bd}) on
the right-hand side of (\ref{e11}), we obtain the following inequality:%
\begin{equation}
\partial _{t}\mathcal{Y}_{k}\left( t\right) +\sum\nolimits_{i\in \mathbb{N}%
_{m}}\gamma _{n_{k}}\int_{\Omega }\left\vert \nabla \left\vert
u_{i}\right\vert ^{\frac{n_{k}+1}{2}}\right\vert ^{2}dx\leq c\left(
n_{k}\right) ^{\sigma _{1}}\left( \mathcal{Y}_{k-1}+1\right) ^{\theta _{k}},
\label{claim}
\end{equation}%
where $c,$ $\sigma _{1}$ are positive constants independent of $k,$ and%
\begin{equation*}
\theta _{k}:=\max (\max_{i}\left\{ z_{i,k}\right\} ,\max_{i}\left\{
l_{i,k}\right\} )\geq 1
\end{equation*}%
is a bounded sequence for all $k.$

We are now ready to prove (\ref{claim2}) using (\ref{claim}). To this end,
let $\zeta \left( s\right) $ be a positive function $\zeta :\mathbb{R}%
_{+}\rightarrow \left[ 0,1\right] $ such that $\zeta \left( s\right) =0$ for
$s\in \left[ 0,t-\mu /n_{k}\right] ,$ $\zeta \left( s\right) =1$ if $s\in %
\left[ t,+\infty \right) $ and $\left\vert d\zeta /ds\right\vert \leq
n_{k}/\mu $, if $s\in \left( t-\mu /n_{k},t\right) $. We define $Z_{k}\left(
s\right) =\zeta \left( s\right) \mathcal{Y}_{k}\left( s\right) $ and notice
that%
\begin{equation*}
\frac{d}{ds}Z_{k}\left( s\right) \leq \frac{n_{k}}{\mu }\mathcal{Y}%
_{k}\left( s\right) +\zeta \left( s\right) \frac{d}{ds}\mathcal{Y}_{k}\left(
s\right) .
\end{equation*}%
Combining this estimate with (\ref{e11}), (\ref{bulk}), (\ref{bd}) and
noticing that $Z_{k}\leq \mathcal{Y}_{k}$, we deduce the following estimate
for $Z_{k}$:%
\begin{equation}
\frac{d}{ds}Z_{k}\left( s\right) +C\left( \mu \right) n_{k}Z_{k}\left(
s\right) \leq M_{k}\left( t,\mu \right) ,\text{ for all }s\in \left[ t-\mu
/n_{k},+\infty \right) ,  \label{e12}
\end{equation}%
for some positive constant $C$ independent of $k$. Integrating (\ref{e12})
with respect to $s$ from $t-\mu /n_{k}$ to $t$ and taking into account the
fact that $Z_{k}\left( t-\mu /n_{k}\right) =0,$ we obtain that%
\begin{equation*}
\mathcal{Y}_{k}\left( t\right) =Z_{k}\left( t\right) \leq M_{k}\left( t,\mu
\right) \left( 1-e^{-C\mu }\right) ,
\end{equation*}%
which proves the claim (\ref{claim2}).

\noindent \textbf{Step 3} (The iterative argument). Let now $\tau ^{^{\prime
}}>\tau >0$ be given; define $\mu =2(\tau ^{^{\prime }}-\tau ),$ $t_{0}=\tau
^{^{\prime }}$ and $t_{k}=t_{k-1}-\mu /n_{k},$ $k\geq 1$. Using (\ref{claim2}%
), we have%
\begin{equation}
\sup_{t\geq t_{k-1}}\mathcal{Y}_{k}\left( t\right) \leq c_{0}\left(
n_{k}\right) ^{\sigma }(\sup_{s\geq t_{k}}\mathcal{Y}_{k-1}\left( s\right)
+1)^{\theta _{k}},\text{ }k\geq 1.  \label{e13}
\end{equation}%
Here $c_{0}=c_{0}\left( \mu \right) $ depends only on $\mu $. Now let us
define%
\begin{equation}
\overline{C}:=\sup_{s\geq t_{1}=\tau }\left( \mathcal{Y}_{0}\left( s\right)
+1\right) =\sup_{s\geq t_{1}=\tau }\left( \left\Vert \overrightarrow{u}%
\left( s\right) \right\Vert _{\mathcal{X}^{1}}+1\right) .  \label{e13bis}
\end{equation}%
Thus, we can iterate in (\ref{e13}) with respect to $k\geq 1$ and obtain that%
\begin{align}
\sup_{t\geq t_{k-1}}\mathcal{Y}_{k}\left( t\right) & \leq \left(
c_{0}n_{k}^{\sigma }\right) \left( c_{0}n_{k-1}^{\sigma }\right) ^{\theta
_{k}}\left( c_{0}n_{k-2}^{\sigma }\right) ^{\theta _{k}\theta _{k-1}}\cdot
...\cdot \left( c_{0}n_{0}^{\sigma }\right) ^{\theta _{k}\theta
_{k-1}...\theta _{0}}(\overline{C})^{\xi _{k}}  \label{e13tris} \\
& \leq c_{0}^{A_{k}}2^{\sigma B_{k}}\left( \overline{C}\right) ^{\xi _{k}},
\notag
\end{align}%
where $\xi _{k}:=\theta _{k}\theta _{k-1}...\theta _{0},$ and%
\begin{equation}
A_{k}:=1+\theta _{k}+\theta _{k}\theta _{k-1}+...+\theta _{k}\theta
_{k-1}...\theta _{0},  \label{ak2}
\end{equation}%
\begin{equation}
B_{k}:=k+\theta _{k}\left( k-1\right) +\theta _{k}\theta _{k-1}\left(
k-2\right) +...+\theta _{k}\theta _{k-1}...\theta _{0}.  \label{bk2}
\end{equation}%
We can easily show that%
\begin{equation}
A_{k}\leq \left( c_{1}+n_{k}\right) \sum_{j=1}^{\infty }\frac{1}{c_{1}+n_{j}}%
\text{ and }B_{k}\leq \left( c_{2}+n_{k}\right) \sum_{j=1}^{\infty }\frac{j}{%
c_{2}+n_{j}},  \label{abk}
\end{equation}%
for some positive constants $c_{1},c_{2}$ independent of $k,$ $\mu $.
Therefore, since%
\begin{equation}
\sup_{t\geq t_{0}}\mathcal{Y}_{k}\left( t\right) \leq \sup_{t\geq t_{k-1}}%
\mathcal{Y}_{k}\left( t\right) \leq c_{0}^{A_{k}}2^{\sigma B_{k}}\left(
\overline{C}\right) ^{\xi _{k}}  \label{e14}
\end{equation}%
and the series in (\ref{abk})\ are convergent, we can take the $1+n_{k}$%
-root on both sides of (\ref{e14}) and let $k\rightarrow +\infty $. We deduce%
\begin{equation*}
\sup_{t\geq t_{0}=\tau ^{^{\prime }}}\left\Vert \overrightarrow{u}\left(
t\right) \right\Vert _{\mathcal{X}^{\infty }}\leq \lim_{k\rightarrow +\infty
}\sup_{t\geq t_{0}}\left( \mathcal{Y}_{k}\left( t\right) \right) ^{1/\left(
1+n_{k}\right) },
\end{equation*}%
which, on account of (\ref{e14}), yields%
\begin{equation}
\sup_{t\geq t_{0}=\tau ^{^{\prime }}}\left\Vert \overrightarrow{u}\left(
t\right) \right\Vert _{\mathcal{X}^{\infty }}\leq C\left( \mu \right) \left(
\overline{C}\right) ^{1/c_{3}},  \label{linff}
\end{equation}%
for some positive constant $c_{3}$ independent of $t,$ $k$, $\overrightarrow{%
u},$ $\epsilon ,$ $\mu ,$ and initial data. Note that $\overline{C}$ depends
on $\tau $ (see (\ref{e13bis})).

\noindent \textbf{Step 4} (The final argument). Let us first assume that
\textbf{Property} \textbf{P}$\left( 1\right) $-(i) holds. Then we already
know that the $\mathcal{X}^{1}$-norm of $\overrightarrow{u}\left( t\right) $
is bounded independently of the initial data, for each $t\geq \tau $.
Therefore, from (\ref{linff}) we also obtain the claim for the $\mathcal{X}%
^{\infty }$-norm of $\overrightarrow{u}\left( t\right) $, i.e., \textbf{%
Property} \textbf{P}$\left( \infty \right) $-(i) holds as well. If, on the
other hand, \textbf{Property} \textbf{P}$\left( 1\right) $-(ii) holds, we
can choose $\tau ^{^{\prime }}=\tau +2\mu $ with $\mu =1$ so that $\overline{%
C}$ and $C\left( \mu \right) $ are bounded uniformly with respect to initial
data as $\tau \rightarrow \infty $. Hence, \textbf{Property} \textbf{P}$%
\left( \infty \right) $-(ii) is also satisfied by letting $\tau \rightarrow
\infty $ in (\ref{linff}). In order to show the final property (iii), taking
advantage of the fact that the initial data $\overrightarrow{u}_{0}\in
\mathcal{X}^{\infty }$, it suffices to note that in place of the inequality (%
\ref{claim2}), we may use instead the inequality%
\begin{equation*}
\mathcal{Y}_{k}\left( t\right) \leq Q(\left\Vert \overrightarrow{u}%
_{0}\right\Vert _{\mathcal{X}^{\infty }},\sup_{t>0}M_{k}\left( t,\mu \right)
),
\end{equation*}%
which is an immediate consequence of (\ref{claim}). Arguing analogously as
in \cite[Lemma 5.5.3]{Ma}, we obtain the claim. The proof of Theorem \ref%
{linff} is now complete.

\begin{remark}
\label{rem}It was proven in \cite[Section 5]{Ma} that maximal $L^{p}$%
-regularity for (\ref{s1c})-(\ref{s4c}) can be used to reduce the question
of global existence of the solutions $\overrightarrow{u}^{\epsilon }$ in a
space of maximal regularity, to the boundedness of $\overrightarrow{u}%
^{\epsilon }$\ in a H\"{o}lder norm $C^{0,\beta }\left( \overline{\Omega }%
\right) $, $\beta >0$. It should be possible to prove, under the natural
assumptions of Theorem \ref{linf}, that every classical solution of problem %
\eqref{s1c}-\eqref{s4c} is globally H\"{o}lder continuous on $\overline{%
\Omega }$. Establishing global H\"{o}lder continuity for solutions to
systems with dynamic boundary conditions requires a more detailed analysis,
involving careful local estimates of the solution near the boundary. Of
course, as in the case of Dirichlet/Robin boundary conditions for (\ref{s1c}%
) (see, e.g., \cite{Du2}), these H\"{o}lder bounds should apriori depend on
the $\,L^{\infty }$-norm of the solution. Thus, our analysis constitutes
only the first step in proving boundedness in H\"{o}lder norm $C^{0,\beta
}\left( \overline{\Omega }\right) $. This question remains open for now.
\end{remark}

\subsection{Proof of Theorem \protect\ref{linf3}}

We shall divide the proof into several steps. As in Section 4.1, we can
justify our computations by exploiting the approximation scheme (\ref{s1c})-(%
\ref{s4c}). As before, $c$ will denote a positive constant that is
independent of $t,$ $\epsilon $, $n,$ $\overrightarrow{u}$ and initial data,
which only depends on the other structural parameters of the problem. Such a
constant may vary even from line to line. Without loss of generality, we may
assume that $\delta _{i}=1$, for all $i\in J_{m}=\mathbb{N}_{m}.$

Let $T,\tau $ and $L$ be positive numbers such that $T-2\tau >0$ and $L\geq
1.$ We set $t_{0}=T-2\tau $ and define the sequences%
\begin{equation*}
t_{n}=t_{n-1}+\frac{\tau }{2^{n}},\text{ }k_{n}=L\left( 2-\frac{1}{2^{n}}%
\right) ,\text{ for all }n\geq 1.
\end{equation*}%
Consider the (smooth)\ cut-off\ functions $\eta _{n}\in C^{1}\left( \mathbb{R%
},\left[ 0,1\right] \right) $ with the property that%
\begin{equation*}
\eta _{n}\left( t\right) =\left\{
\begin{array}{ll}
1, & t\geq t_{n,} \\
0, & t<t_{n-1}.%
\end{array}%
\right.
\end{equation*}%
Next, denote $Q_{n}:=I_{n}\times \overline{\Omega }$, where $I_{n}:=\left[
t_{n-1},T\right] $, and the sets%
\begin{equation*}
A_{i,n}^{\Omega }:=\left\{ \left( x,t\right) \in I_{n}\times \Omega
:u_{i}\left( x,t\right) >k_{n}\right\} ,\text{ }A_{i,n}^{\Gamma }:=\left\{
\left( x,t\right) \in I_{n}\times \Gamma :u_{i}\left( x,t\right)
>k_{n}\right\} .
\end{equation*}%
Let $\overline{A}_{i,n}=A_{i,n}^{\Omega }\cup A_{i,n}^{\Gamma }$, and note
that%
\begin{equation*}
\overline{A}_{i,n}=\left\{ \left( x,t\right) \in Q_{n}:u_{i}\left(
x,t\right) >k_{n}\right\} .
\end{equation*}%
Finally, we denote by $\left\vert A_{i,n}^{\Omega }\right\vert $ the ($N+1$%
-dimensional) Lebesgue measure of the set $A_{i,n}^{\Omega }$, and by $%
\left\vert A_{i,n}^{\Gamma }\right\vert $, the ($N$-dimensional) Lebesgue
measure of $A_{i,n}^{\Gamma }$, respectively. We note that, according to (%
\ref{2.1b})-(\ref{2.1t}), we have%
\begin{equation*}
\left\vert Q_{n}\right\vert =\left\vert I_{n}\right\vert \mu \left(
\overline{\Omega }\right) =\left\vert I_{n}\right\vert \left( \left\vert
\Omega \right\vert +\left\vert \Gamma \right\vert \right) ,
\end{equation*}%
and we can do so similarly for the set $\overline{A}_{i,n}$.

\noindent \textbf{Step 1}. (The energy inequality). We define the truncated
functions
\begin{equation*}
u_{i,n}\left( x,t\right) :=\max \left\{ u_{i}\left( x,t\right)
-k_{n},0\right\} =\left( u_{i}-k_{n}\right) _{+}.
\end{equation*}%
We begin by multiplying equation (\ref{s1c}) by $u_{i,n}\eta _{n}^{2}\left(
t\right) $ and integrating the resulting identity over $I_{n}\times \Omega $%
. Then, we multiply (\ref{s3c}) by $u_{i,n}\eta _{n}^{2}\left( t\right) $
and integrate over $I_{n}\times \Gamma $. Adding as usual (cf., e.g.,
Section 4.1), then exploiting the growth assumptions (\ref{ga}) on $f_{i}$
and $g_{i}$, the fact that $\left\vert \eta _{n}^{^{\prime }}\left( t\right)
\right\vert \leq 2^{n}/\tau $, we obtain after standard transformations%
\begin{align}
& \max_{t\in I_{n}}\left( \int_{\Omega }u_{i,n}^{2}\left( x,t\right)
dx+\int_{\Gamma }u_{i,n}^{2}\left( x,t\right) dS\right)
+\iint\nolimits_{I_{n}\times \Omega }a_{i}\left( u_{i}\right) \left\vert
\nabla \left( u_{i,n}\eta _{n}\right) \left( x,t\right) \right\vert ^{2}dxdt
\label{d1} \\
& \leq c\iint\nolimits_{A_{i,n}^{\Omega }}\left( \frac{2^{n}}{\tau }%
u_{i,n}^{2}\left( x,t\right) \eta _{n}+\sum\nolimits_{j\in \mathbb{N}%
_{m}}\left\vert u_{j}\left( x,t\right) \right\vert ^{\theta }u_{i,n}\left(
x,t\right) \eta _{n}^{2}\left( t\right) +u_{i,n}\left( x,t\right) \eta
_{n}^{2}\right) dxdt  \notag \\
& +c\iint\nolimits_{A_{i,n}^{\Gamma }}\left( \frac{2^{n}}{\tau }%
u_{i,n}^{2}\left( x,t\right) \eta _{n}+\sum\nolimits_{j\in \mathbb{N}%
_{m}}\left\vert u_{j}\left( x,t\right) \right\vert ^{\beta }u_{i,n}\left(
x,t\right) \eta _{n}^{2}+u_{i,n}\left( x,t\right) \eta _{n}^{2}\right) dSdt,
\notag
\end{align}%
where we have set $\theta :=\max_{i\in \mathbb{N}_{m}}\theta _{i}$ and $%
\beta :=\max_{i\in \mathbb{N}_{m}}\beta _{i}$. Now, we wish to estimate the
terms on the right-hand side of (\ref{d1}). To this end, set%
\begin{equation*}
\mathcal{A}_{n}^{\Omega }:=\cup _{k\geq 1}A_{k,n}^{\Omega }\text{, }\mathcal{%
A}_{n}^{\Gamma }:=\cup _{k\geq 1}A_{k,n}^{\Gamma }
\end{equation*}%
and note that on $\mathcal{A}_{n}^{\Omega }\backslash A_{i,n}^{\Omega }$ and
$\mathcal{A}_{n}^{\Gamma }\backslash A_{i,n}^{\Gamma }$, respectively, we
have $u_{i}\left( x,t\right) \leq k_{n}\leq 2L$ and $u_{i}\left( x,t\right)
_{\mid \Gamma }\leq k_{n}\leq 2L$, respectively. Therefore,
\begin{equation}
\iint\nolimits_{\mathcal{A}_{n}^{\Omega }\backprime A_{i,n}^{\Omega
}}\left\vert u_{i}\right\vert ^{\theta +1}dxdt\leq cL^{\theta
+1}\iint\nolimits_{\mathcal{A}_{n}^{\Omega }\backprime A_{i,n}^{\Omega
}}\left( 1\right) dxdt\leq cL^{\theta +1}\sum\nolimits_{j\in \mathbb{N}%
_{m}}\left\vert A_{j,n}^{\Omega }\right\vert ,  \label{d2}
\end{equation}%
and, analogously, for the trace of $u_{i}$ we have%
\begin{equation}
\iint\nolimits_{\mathcal{A}_{n}^{\Gamma }\backprime A_{i,n}^{\Gamma
}}\left\vert u_{i}\right\vert ^{\beta +1}dSdt\leq cL^{\beta
+1}\iint\nolimits_{\mathcal{A}_{n}^{\Gamma }\backprime A_{i,n}^{\Gamma
}}\left( 1\right) dSdt\leq cL^{\beta +1}\sum\nolimits_{j\in \mathbb{N}%
_{m}}\left\vert A_{j,n}^{\Gamma }\right\vert ,  \label{d2b}
\end{equation}%
Since $k_{n}\sim L$, as $n\rightarrow \infty $, it is easy to see that the
following inequalities hold:%
\begin{equation*}
\left\{
\begin{array}{l}
L^{\alpha +1}\left\vert A_{j,n}^{\Omega }\right\vert \leq ck_{n}^{\alpha
+1}\left\vert A_{j,n}^{\Omega }\right\vert \leq
c\iint\nolimits_{A_{j,n}^{\Omega }}\left\vert u_{j}\right\vert ^{\alpha
+1}dxdt, \\
L^{\beta +1}\left\vert A_{j,n}^{\Gamma }\right\vert \leq ck_{n}^{\beta
+1}\left\vert A_{j,n}^{\Gamma }\right\vert \leq
c\iint\nolimits_{A_{j,n}^{\Gamma }}\left\vert u_{j}\right\vert ^{\beta
+1}dSdt.%
\end{array}%
\right.
\end{equation*}%
From these estimates, we thus find that%
\begin{align}
\iint\nolimits_{I_{n}\times \Omega }\left\vert u_{j}\right\vert ^{\theta
}u_{i,n}dxdt& \leq \iint\nolimits_{\mathcal{A}_{n}^{\Omega }}\left(
\left\vert u_{j}\right\vert ^{\theta +1}+\left\vert u_{i}\right\vert
^{\theta +1}\right) dxdt  \label{d4} \\
& \leq c\sum\nolimits_{j\in \mathbb{N}_{m}}\iint\nolimits_{A_{j,n}^{\Omega
}}\left\vert u_{j}\right\vert ^{\theta +1}dxdt  \notag
\end{align}%
and, similarly,%
\begin{align}
\iint\nolimits_{I_{n}\times \Gamma }\left\vert u_{j}\right\vert ^{\beta
}u_{i,n}dSdt& \leq \iint\nolimits_{\mathcal{A}_{n}^{\Gamma }}\left(
\left\vert u_{j}\right\vert ^{\beta +1}+\left\vert u_{i}\right\vert ^{\beta
+1}\right) dxdt  \label{d4b} \\
& \leq c\sum\nolimits_{j\in \mathbb{N}_{m}}\iint\nolimits_{A_{j,n}^{\Gamma
}}\left\vert u_{j}\right\vert ^{\beta +1}dSdt.  \notag
\end{align}%
Hence, using the above inequalities (\ref{d4})-(\ref{d4b}) on the right-hand
side of (\ref{d1}), and summing the resulting relation over $i\in \mathbb{N}%
_{m}$, we deduce%
\begin{align}
& \max_{t\in I_{n}}\left( \sum\nolimits_{i\in \mathbb{N}_{m}}\int_{\Omega
}u_{i,n}^{2}\left( x,t\right) dx+\sum\nolimits_{i\in \mathbb{N}%
_{m}}\int_{\Gamma }u_{i,n}^{2}\left( x,t\right) dS\right)  \label{ei} \\
& +\sum\nolimits_{i\in \mathbb{N}_{m}}\iint\nolimits_{I_{n}\times \Omega
}\alpha _{i}\left\vert u_{i}\right\vert ^{p_{i}}\left\vert \nabla \left(
u_{i,n}\eta _{n}\right) \right\vert ^{2}dxdt  \notag \\
& \leq \frac{2^{n}c}{\tau }\sum\nolimits_{i\in \mathbb{N}_{m}}\iint%
\nolimits_{A_{i,n}^{\Omega }}\left\vert u_{i}\left( x,t\right) \right\vert
^{\delta }dxdt+\frac{2^{n}c}{\tau }\sum\nolimits_{i\in \mathbb{N}%
_{m}}\iint\nolimits_{A_{i,n}^{\Gamma }}\left\vert u_{i}\left( x,t\right)
\right\vert ^{\gamma }dSdt,  \notag
\end{align}%
where $\delta $ and $\gamma $ are defined as in (\ref{exp}). Here we have
also used assumption (\ref{A1}).

\noindent \textbf{Step 2}. (Additional estimates). From the definition of $%
k_{n}$, we see that $1-k_{n}/k_{n+1}\geq 2^{-\left( n+2\right) }$, which
yields%
\begin{align}
\iint\nolimits_{A_{i,n+1}^{\Omega }}\left\vert u_{i}\right\vert ^{\delta
}dxdt& \leq 2^{\left( n+2\right) \delta }\iint\nolimits_{A_{i,n+1}^{\Omega
}}\left\vert u_{i}\right\vert ^{\delta }\left( 1-\frac{k_{n}}{k_{n+1}}%
\right) dxdt  \label{d5} \\
& \leq 2^{n\delta }c\iint\nolimits_{A_{i,n+1}^{\Omega }}\left(
u_{i}-k_{n}\right) _{+}^{\delta }dxdt.  \notag
\end{align}%
Moreover, the same argument gives%
\begin{equation}
\iint\nolimits_{A_{i,n+1}^{\Gamma }}\left\vert u_{i}\right\vert ^{\gamma
}dSdt\leq 2^{n\gamma }c\iint\nolimits_{A_{i,n+1}^{\Gamma }}\left(
u_{i}-k_{n}\right) _{+}^{\gamma }dSdt.  \label{d5b}
\end{equation}%
On the other hand, since on $A_{i,n+1}^{\Omega }\cup A_{i,n+1}^{\Gamma },$
we have $\left( u_{i}-k_{n}\right) _{+}\geq k_{n+1}-k_{n}$, there holds%
\begin{align}
\iint\nolimits_{A_{i,n+1}^{\Omega }}\left( u_{i}-k_{n}\right) _{+}^{\delta
}dxdt& \geq \left( k_{n+1}-k_{n}\right) ^{\delta }\left\vert
A_{i,n+1}^{\Omega }\right\vert \geq c\frac{L^{\delta }}{2^{n\delta }}%
\left\vert A_{i,n+1}^{\Omega }\right\vert ,\text{ }  \label{d6} \\
\iint\nolimits_{A_{i,n+1}^{\Gamma }}\left( u_{i}-k_{n}\right) _{+}^{\gamma
}dSdt& \geq \left( k_{n+1}-k_{n}\right) ^{\gamma }\left\vert
A_{i,n+1}^{\Gamma }\right\vert \geq c\frac{L^{\gamma }}{2^{n\gamma }}%
\left\vert A_{i,n+1}^{\Gamma }\right\vert .  \notag
\end{align}%
Because of these two inequalities (\ref{d6}), for any positive number $%
\lambda $ such that, if $\lambda <\delta $ and $\lambda <\gamma $, on
account of Holder's inequality, it also holds%
\begin{align}
\iint\nolimits_{A_{i,n+1}^{\Omega }}\left( u_{i}-k_{n+1}\right)
_{+}^{\lambda }dxdt& \leq \left( \iint\nolimits_{A_{i,n+1}^{\Omega }}\left(
u_{i}-k_{n+1}\right) _{+}^{\delta }dxdt\right) ^{\lambda /\delta }\left\vert
A_{i,n+1}^{\Omega }\right\vert ^{1-\lambda /\delta }  \label{d7} \\
& \leq \frac{c2^{n\left( \delta -\lambda \right) }}{L^{\delta -\lambda }}%
\iint\nolimits_{A_{i,n+1}^{\Omega }}\left( u_{i}-k_{n}\right) _{+}^{\delta
}dxdt,  \notag
\end{align}%
and%
\begin{align}
\iint\nolimits_{A_{i,n+1}^{\Gamma }}\left( u_{i}-k_{n+1}\right)
_{+}^{\lambda }dSdt& \leq \left( \iint\nolimits_{A_{i,n+1}^{\Gamma }}\left(
u_{i}-k_{n+1}\right) _{+}^{\gamma }dxdt\right) ^{\lambda /\gamma }\left\vert
A_{i,n+1}^{\Gamma }\right\vert ^{1-\lambda /\gamma }  \label{d7b} \\
& \leq \frac{c2^{n\left( \gamma -\lambda \right) }}{L^{\gamma -\lambda }}%
\iint\nolimits_{A_{i,n+1}^{\Gamma }}\left( u_{i}-k_{n}\right) _{+}^{\gamma
}dSdt.  \notag
\end{align}

Next, we will collect some useful inequalities which follow from the
following well-known embeddings: $H^{1}\left( \Omega \right) \subset
L^{p_{s}},$ $p_{s}:=2N/\left( N-2\right) $ and $H^{1}\left( \Omega \right)
\subset L^{q_{s}}\left( \Gamma \right) ,$ $q_{s}:=2\left( N-1\right) /\left(
N-2\right) $. We give the argument for $N>2$, the case $N\leq 2$ can be
treated analogously. Suppressing the dependance on the subscript $n$ for the
moment, from H\"{o}lder's inequality and these Sobolev embeddings, for every
$v^{M_{i}}\in W^{1,2}\left( I\times \Omega \right) $ we have%
\begin{align}
\iint\nolimits_{^{I\times \Omega }}v^{s}dxdt& \leq \int\nolimits_{I}\left(
\int\nolimits_{\Omega }v^{M_{i}p_{s}}dx\right) ^{\frac{N-2}{N}}\left(
\int\nolimits_{\Omega }v^{2}dx\right) ^{\frac{2}{N}}dt  \label{d8} \\
& \leq \left( \iint\nolimits_{^{I\times \Omega }}\left( \left\vert \nabla
v^{M_{i}}\right\vert ^{2}+\left\vert v\right\vert ^{M_{i}}\right)
dxdt\right) \times \left( \max_{t\in I}\int\nolimits_{\Omega }v^{2}dx\right)
^{\frac{2}{N}},  \notag
\end{align}%
where, for each $M_{i}>0$, we have set%
\begin{equation*}
s=2M_{i}+\frac{4}{N}.
\end{equation*}%
Similarly, for each $M_{i}>0$ and $l=2M_{i}+2/\left( N-1\right) ,$ we have%
\begin{align}
\iint\nolimits_{^{I\times \Gamma }}v^{s}dS& \leq \int\nolimits_{I}\left(
\int\nolimits_{\Gamma }v^{M_{i}q_{s}}dS\right) ^{\frac{N-2}{N-1}}\left(
\int\nolimits_{\Gamma }v^{2}dS\right) ^{\frac{1}{N-1}}dt  \label{d8b} \\
& \leq \left( \iint\nolimits_{I\times \Omega }\left( \left\vert \nabla
v^{M_{i}}\right\vert ^{2}+\left\vert v\right\vert ^{M_{i}}\right)
dxdt\right) \times \left( \max_{t\in I}\int\nolimits_{\Gamma }v^{2}dS\right)
^{\frac{1}{N-1}}.  \notag
\end{align}%
Exploiting (\ref{d8})-(\ref{d8b}) with $M_{i}=p_{i}/2+1,$ $I=I_{n+1}$, $%
v=\left( u_{i}-k_{n+1}\right) _{+}$, we get%
\begin{align}
& \iint\nolimits_{I_{n+1}\times \Omega }\left( u_{i}-k_{n+1}\right)
_{+}^{s}dxdt  \label{d9} \\
& \leq \left( \iint\nolimits_{^{I_{n+1}\times \Omega }}\left( \left\vert
u_{i}\right\vert ^{p_{i}}\left\vert \nabla u_{i,n+1}\right\vert
^{2}+u_{i,n+1}^{M_{i}}\right) dxdt\right) \times \left( \max_{t\in
I_{n+1}}\int\nolimits_{\Omega }u_{i,n+1}^{2}dx\right) ^{\frac{2}{N}}  \notag
\end{align}%
and%
\begin{align}
& \iint\nolimits_{^{I_{n+1}\times \Gamma }}\left( u_{i}-k_{n+1}\right)
_{+}^{l}dSdt  \label{d9b} \\
& \leq \left( \iint\nolimits_{^{I_{n+1}\times \Omega }}\left( \left\vert
u_{i}\right\vert ^{p_{i}}\left\vert \nabla u_{i,n+1}\right\vert
^{2}+u_{i,n+1}^{M_{i}}\right) dxdt\right) \times \left( \max_{t\in
I_{n+1}}\int\nolimits_{\Gamma }u_{i,n+1}^{2}dx\right) ^{\frac{1}{N-1}}.
\notag
\end{align}%
Finally, from (\ref{ei}) we see that estimates (\ref{d5})-(\ref{d5b}) yield
the following inequality%
\begin{align}
& \max_{t\in I_{n}}\left( \sum\nolimits_{i\in \mathbb{N}_{m}}\int_{\Omega
}u_{i,n}^{2}\left( x,t\right) dx+\sum\nolimits_{i\in \mathbb{N}%
_{m}}\int_{\Gamma }u_{i,n}^{2}\left( x,t\right) dS\right)  \label{ei2} \\
& +\sum\nolimits_{i\in \mathbb{N}_{m}}\iint\nolimits_{I_{n}\times \Omega
}\alpha _{i}\left\vert u_{i}\right\vert ^{p_{i}}\left\vert \nabla \left(
u_{i,n}\eta _{n}\right) \right\vert ^{2}dxdt  \notag \\
& \leq \frac{2^{n\left( \delta +1\right) }c}{\tau }\sum\nolimits_{i\in
\mathbb{N}_{m}}\iint\nolimits_{A_{i,n}^{\Omega }}\left( u_{i}-k_{n-1}\right)
_{+}^{\delta }dxdt  \notag \\
& +\frac{2^{n\left( \gamma +1\right) }c}{\tau }\sum\nolimits_{i\in \mathbb{N}%
_{m}}\iint\nolimits_{A_{i,n}^{\Gamma }}\left( u_{i}-k_{n-1}\right)
_{+}^{\gamma }dSdt.  \notag
\end{align}

\noindent \textbf{Step 3}. (The iterative argument). We continue our main
argument by first recalling the following result (see, e.g., \cite[Lemma 4.1]%
{DiB}).

\begin{lemma}
\label{seq} Let $\left\{ \mathcal{Y}_{n}\right\} $ be a sequence of positive
numbers such that it satisfies%
\begin{equation}
\mathcal{Y}_{n+1}\leq Cb^{n}\mathcal{Y}_{n}^{1+\kappa },  \label{rec}
\end{equation}%
for some constants $C,b,\kappa >0$. If $\mathcal{Y}_{0}\leq C^{-1/\kappa
}b^{-1/\kappa ^{2}}$, then $\mathcal{Y}_{n}\rightarrow 0$ as $n\rightarrow
\infty .$
\end{lemma}

Define%
\begin{equation*}
\mathcal{Y}_{i,n}:=\frac{1}{\left\vert Q_{n}\right\vert }\left(
\iint\nolimits_{I_{n}\times \Omega }\left( u_{i}-k_{n}\right) _{+}^{\delta
}dxdt+\iint\nolimits_{I_{n}\times \Gamma }\left( u_{i}-k_{n}\right)
_{+}^{\gamma }dSdt\right) ,
\end{equation*}%
where we recall that $Q_{n}=I_{n}\times \overline{\Omega }$ and
\begin{equation*}
\left\vert Q_{n}\right\vert =\left\vert I_{n}\times \Omega \right\vert
+\left\vert I_{n}\times \Gamma \right\vert .
\end{equation*}%
Set $\mathcal{Y}_{n}=\sum\nolimits_{i=1}^{m}\mathcal{Y}_{i,n}$. The goal now
is to show that the sequence $\left\{ \mathcal{Y}_{n}\right\} $ satisfies a
recursive relation of the form (\ref{rec}). First, using the definition of $%
\mathcal{Y}_{n},$ we can rewrite (\ref{ei2}) as the following inequality:%
\begin{align}
& \max_{t\in I_{n+1}}\left( \sum\nolimits_{i\in \mathbb{N}_{m}}\int_{\Omega
}u_{i,n+1}^{2}\left( x,t\right) dx+\sum\nolimits_{i\in \mathbb{N}%
_{m}}\int_{\Gamma }u_{i,n+1}^{2}\left( x,t\right) dS\right)  \label{ei3} \\
& +\sum\nolimits_{i\in \mathbb{N}_{m}}\iint\nolimits_{I_{n+1}\times \Omega
}\alpha _{i}\left\vert u_{i}\right\vert ^{p_{i}}\left\vert \nabla \left(
u_{i,n+1}\eta _{n+1}\right) \right\vert ^{2}dxdt  \notag \\
& \leq c2^{\left( n+1\right) \left( \rho +1\right) }\tau ^{-1}\left\vert
Q_{n+1}\right\vert \mathcal{Y}_{n},  \notag
\end{align}%
for $\rho :=\max \left( \gamma ,\delta \right) >1.$ Secondly, applying (\ref%
{d7}) to the bulk integral over $u_{i,n+1}^{M_{i}}$ (where $M_{i}:=p_{i}/2+1$%
), which occurs in the integrals in (\ref{d9})-(\ref{d9b}), and then using (%
\ref{ei3}) to estimate the second terms in those products, we obtain%
\begin{align}
\iint\nolimits_{^{I_{n+1}\times \Omega }}& \left( u_{i}-k_{n+1}\right)
_{+}^{s}dxdt  \label{d10} \\
& \leq c\left( \frac{2^{n\rho }}{\tau }+\frac{2^{n\left( \delta
-M_{i}\right) }}{L^{\delta -M_{i}}}\right) \left\vert Q_{n+1}\right\vert
\mathcal{Y}_{n}\left( \frac{c2^{n\rho }}{\tau }\left\vert Q_{n+1}\right\vert
\mathcal{Y}_{n}\right) ^{\frac{2}{N}},  \notag
\end{align}%
and%
\begin{align}
& \iint\nolimits_{^{I_{n+1}\times \Gamma }}\left( u_{i}-k_{n+1}\right)
_{+}^{l}dSdt  \label{d10b} \\
& \leq c\left( \frac{2^{n\rho }}{\tau }+\frac{2^{n\left( \gamma
-M_{i}\right) }}{L^{\gamma -M_{i}}}\right) \left\vert Q_{n+1}\right\vert
\mathcal{Y}_{n}\left( \frac{c2^{n\rho }}{\tau }\left\vert Q_{n+1}\right\vert
\mathcal{Y}_{n}\right) ^{\frac{1}{N-1}}.  \notag
\end{align}%
H\"{o}lder's inequality applied to $\mathcal{Y}_{i,n+1}$ yields%
\begin{align}
\mathcal{Y}_{i,n+1}& \leq \frac{1}{\left\vert Q_{n+1}\right\vert }\left(
\iint\nolimits_{^{I_{n+1}\times \Omega }}\left( u_{i}-k_{n+1}\right)
_{+}^{s}dxdt\right) ^{\delta /s}\left\vert A_{i,n+1}^{\Omega }\right\vert
^{1-\delta /s}  \label{d10t} \\
& +\frac{1}{\left\vert Q_{n+1}\right\vert }\left(
\iint\nolimits_{^{I_{n+1}\times \Gamma }}\left( u_{i}-k_{n+1}\right)
_{+}^{l}dSdt\right) ^{\gamma /l}\left\vert A_{i,n+1}^{\Gamma }\right\vert
^{1-\gamma /l}.  \notag
\end{align}%
Inserting the estimates for $\left\vert A_{i,n+1}^{\Omega }\right\vert $ and
$\left\vert A_{i,n+1}^{\Gamma }\right\vert $, respectively, from (\ref{d6}),
on the right-hand side of (\ref{d10t}), we deduce%
\begin{align*}
\mathcal{Y}_{i,n+1}& \leq c\left\vert Q_{n+1}\right\vert ^{\frac{2\delta }{Ns%
}}\mathcal{Y}_{n}^{1+\frac{2\delta }{Ns}}\left( \frac{2^{n\rho }}{\tau }+%
\frac{2^{n\left( \delta -M_{i}\right) }}{L^{\delta -M_{i}}}\right) ^{\delta
/s} \\
& \times \left( \frac{c2^{n\rho }}{\tau }\right) ^{\frac{2\delta }{Ns}%
}\left( \frac{2^{n\delta }}{L^{\delta }}\right) ^{1-\delta /s} \\
& +c\left\vert Q_{n+1}\right\vert ^{\frac{\gamma }{\left( N-1\right) l}}%
\mathcal{Y}_{n}^{1+\frac{\gamma }{\left( N-1\right) l}}\left( \frac{2^{n\rho
}}{\tau }+\frac{2^{n\left( \gamma -M_{i}\right) }}{L^{\gamma -M_{i}}}\right)
^{\gamma /l} \\
& \times \left( \frac{c2^{n\rho }}{\tau }\right) ^{\frac{\gamma }{\left(
N-1\right) l}}\left( \frac{2^{n\gamma }}{L^{\gamma }}\right) ^{1-\gamma /l}.
\end{align*}%
Henceforth, by setting%
\begin{equation*}
\kappa :=\kappa \left( \delta ,\gamma \right) =\left\{
\begin{array}{ll}
\max \left( \frac{2\delta }{Ns},\frac{\gamma }{\left( N-1\right) l}\right) ,
& \text{if }\mathcal{Y}_{n}\geq 1, \\
\min \left( \frac{2\delta }{Ns},\frac{\gamma }{\left( N-1\right) l}\right) ,
& \text{if }\mathcal{Y}_{n}<1,%
\end{array}%
\right.
\end{equation*}%
the above inequality yields the recursive relation%
\begin{equation*}
\mathcal{Y}_{n+1}\leq Cb^{n}\mathcal{Y}_{n}^{1+\kappa },\text{ with }\kappa
>0,
\end{equation*}%
where $C\sim L^{-\sigma }$ depends on $\tau ^{-1}$ and $b\sim 2^{\zeta }$,
for some positive constants $\sigma ,\zeta $ depending on $\delta ,\gamma ,s$%
, $l.$ Therefore, if we choose $L\geq 1$ sufficiently large so there holds%
\begin{equation*}
\mathcal{Y}_{0}\leq C^{-1/\kappa }b^{-1/\kappa ^{2}}\lessapprox L^{\sigma
/\kappa }b^{-1/\kappa ^{2}},
\end{equation*}%
then by Lemma \ref{seq}, it follows that $\mathcal{Y}_{n}\rightarrow 0$ as $%
n\rightarrow \infty $. This implies that%
\begin{equation*}
\sup_{\left( x,t\right) \in \left[ T-\tau ,T\right] \times \overline{\Omega }%
}u_{i}\left( x,t\right) \leq \lim_{n\rightarrow \infty }k_{n}\leq 2L.
\end{equation*}%
In order to estimate $u_{i}\left( x,t\right) $ from below it suffices to
apply the result just obtained to the functions $\widetilde{u}_{i}\left(
x,t\right) =-u_{i}\left( x,t\right) $, which satisfies a system of the same
type as for $u_{i}\left( x,t\right) ,$ with nonlinearities $\widetilde{a}%
_{i}\left( x,t,\widetilde{u}_{i}\right) =-a_{i}\left( x,t,-u_{i}\right) ,$ $%
\widetilde{f}_{i}\left( x,t,\widetilde{u}_{i}\right) =-f_{i}\left(
x,t,-u_{i}\right) $ and $\widetilde{g}_{i}\left( x,t,\widetilde{u}%
_{i}\right) =-g_{i}\left( x,t,-u_{i}\right) $, respectively. These functions
are subject to the same conditions (\ref{A1bis}), (\ref{ga}). This yields
the desired estimate (\ref{llinf}). Finally, we may conclude that if $T$ is
sufficiently large, we can take $\tau =1$ in (\ref{llinf}), which also
immediately gives the first conclusion in the theorem. The proof is finished.

\subsection{Proof of Theorems \protect\ref{diss}, \protect\ref{diss2}}

In order to justify our computations for problem (\ref{s1}), (\ref{s3}), (%
\ref{s4}), (\ref{s5}), it is not clear how to use the scheme introduced at
the beginning of the section due to the nature of the boundary domain
(indeed, $\overline{\Gamma }_{1}\cap \overline{\Gamma }_{2}\neq \varnothing $
in that case, so we \emph{cannot} exploit maximal regularity theory to
construct smooth solutions unless $\Gamma _{2}\equiv \varnothing $).
However, the proof can be based on the application of a Galerkin
approximation scheme which is not standard due to the nature of the boundary
condition (\ref{s3}). We refer the reader to, e.g., \cite{CGGM}, for further
details, where a system of reaction-diffusion equations for the phase-field
equations with dynamic boundary conditions were considered (cf. also \cite%
{GW} in a degenerate case).

We begin the proof of Theorem \ref{diss}\ with a result for the eigenvalue
problem for so-called Wentzell Laplacian $\Delta _{W}$\ (see \cite[Appendix]%
{GalNS}). More precisely, let us consider the equation%
\begin{equation}
-\Delta \varphi =\Lambda \varphi \text{ in }\Omega ,  \label{sp1}
\end{equation}%
with a boundary condition that depends on the eigenvalue $\Lambda $
explicitly,%
\begin{equation}
\partial _{\mathbf{n}}\varphi =\Lambda \varphi \text{ on }\Gamma _{1},
\label{sp2}
\end{equation}%
such that%
\begin{equation}
\partial _{\mathbf{n}}\varphi +\varphi =0\text{ on }\Gamma _{2}.  \label{sp3}
\end{equation}%
(recall that $\Gamma _{2}$ is assumed to be a set of positive measure and
that (\ref{sp3}) holds where a Dirichlet boundary condition for $u=u_{i}$ is
satisfied on $\Gamma _{2}\times \mathbb{R}_{+}$). Such a function $\varphi $
will be called an eigenfunction associated with $\Lambda $ and the set of
all eigenvalues $\Lambda $ of (\ref{sp1})-(\ref{sp3}) will be denoted by $%
\Lambda _{j},$ $j\in \mathbb{N}$. Let $\varphi _{1}\in C^{2}\left( \overline{%
\Omega }\right) $ and $\Lambda _{1}$, denote the principal eigenfunction and
eigenvalue of (\ref{sp1})-(\ref{sp3}), respectively. We have the following.

\begin{proposition}
For the spectral problem (\ref{sp1})-(\ref{sp3}), $\Lambda _{1}>0$ is simple
and $\varphi _{1}>0$ in $\overline{\Omega }$.
\end{proposition}

\begin{proof}
Using the standard characterization for the eigenvalues $\Lambda _{j}$ of $%
\Delta _{W}$ (see, e.g., \cite{GalNS}), we obtain that the following min-max
principle holds:%
\begin{equation}
\Lambda _{j}=\min_{\substack{ Y_{j}\subset H^{1}\left( \Omega \right) ,  \\ %
\dim Y_{j}=j}}\max_{0\neq \varphi \in Y_{j}}R_{W}\left( \varphi ,\varphi
\right) ,\text{ }j\in \mathbb{N}\text{,}  \label{minmax}
\end{equation}%
where the Rayleigh quotient $R_{W}$, for the (boundary perturbed) Wentzell
operators $\Delta _{W}$, is given by%
\begin{equation}
R_{W}\left( \varphi ,\varphi \right) :=\frac{\left\Vert \nabla \varphi
\right\Vert _{2}^{2}+\left\langle \varphi ,\varphi \right\rangle
_{L^{2}\left( \Gamma _{2}\right) }}{\left\Vert \varphi \right\Vert _{\mathbb{%
X}^{2}}^{2}},\text{ }0\neq \varphi \in H^{1}\left( \Omega \right) .
\label{rqw}
\end{equation}%
Exploiting a well-known Friedrichs-Poincare's inequality, we have $%
R_{W}\left( \varphi ,\varphi \right) \geq c_{W}\left\Vert \varphi
\right\Vert _{\mathbb{X}^{2}}^{2},$ for some positive constant $c_{W},$
which implies that $\Lambda _{j}>0$, for any $j\in \mathbb{N}$. By the
maximum principle, $\varphi _{1}$ is positive in $\overline{\Omega }$ since $%
\Gamma _{2}$ has positive surface measure. The fact that $\Lambda _{1}$ is
simple, follows again from the maximum principle (see, e.g., \cite{BF}).
\end{proof}

We are now ready to give the proof of Theorem \ref{diss}. Without loss of
generality, we can take $\delta _{i}=1,$ for all $i\in J_{m}$. We multiply (%
\ref{s1}) by $\left\vert u_{i}\right\vert ^{m_{i}-1}sgn\left( u_{i}\right)
\varphi _{1},$ and integrate over $\Omega $, for each $i\in \mathbb{N}_{m}$.
We obtain%
\begin{eqnarray}
&&\frac{1}{m_{i}}\frac{d}{dt}\int_{\Omega }\left\vert u_{i}\right\vert
^{m_{i}}\varphi _{1}dx+\left\langle f_{i}\left( x,t,\overrightarrow{u}%
\right) ,\left\vert u_{i}\right\vert ^{m_{i}-1}sgn\left( u_{i}\right)
\varphi _{1}\right\rangle _{L^{2}\left( \Omega \right) }  \label{p1} \\
&&-\int_{\Omega }div\left( a_{i}\left( u_{i}\right) \nabla u_{i}\right)
\left\vert u_{i}\right\vert ^{m_{i}-1}sgn\left( u_{i}\right) \varphi _{1}dx
\notag \\
&=&0.  \notag
\end{eqnarray}%
Similarly, we multiply (\ref{s3}) by $\left\vert u_{i}\right\vert
^{m_{i}-1}sgn\left( u_{i}\right) \varphi _{1}$ and integrate the relation
over $\Gamma $ (recall that (\ref{s5}) holds over $\Gamma _{2}$). We have%
\begin{align}
& \frac{1}{m_{i}}\frac{d}{dt}\int_{\Gamma _{1}}\left\vert u_{i}\right\vert
^{m_{i}}\varphi _{1}dS+\int_{\Gamma }a_{i}\left( u_{i}\right) \partial _{%
\mathbf{n}}u_{i}\left\vert u_{i}\right\vert ^{m_{i}-1}sgn\left( u_{i}\right)
\varphi _{1}dS  \label{p2} \\
& +\left\langle g_{i}\left( x,t,\overrightarrow{u}\right) ,\left\vert
u_{i}\right\vert ^{m_{i}-1}sgn\left( u_{i}\right) \varphi _{1}\right\rangle
_{L^{2}\left( \Gamma _{1}\right) }  \notag \\
& =0,  \notag
\end{align}%
for each $i\in \mathbb{N}_{m}$. Consider the following real-valued function%
\begin{equation}
\mathcal{E}\left( \overrightarrow{u}\left( x,t\right) \right)
=\sum\nolimits_{i\in \mathbb{N}_{m}}\frac{1}{m_{i}}\left\vert u_{i}\left(
x,t\right) \right\vert ^{m_{i}}.  \label{energy}
\end{equation}%
Integrating by parts in (\ref{p1}), then using (\ref{p2}), on account of the
following computation%
\begin{align*}
& \int_{\Omega }div\left( a_{i}\left( u_{i}\right) \nabla u_{i}\right)
\left\vert u_{i}\right\vert ^{m_{i}-1}sgn\left( u_{i}\right) \varphi _{1}dx
\\
& =-\left( m_{i}-1\right) \int_{\Omega }a_{i}\left( u_{i}\right) \left\vert
\nabla u_{i}\right\vert ^{2}\left\vert u_{i}\right\vert ^{m_{i}-2}\varphi
_{1}dx \\
& -\int_{\Omega }a_{i}\left( u_{i}\right) \left\vert u_{i}\right\vert
^{m_{i}-1}sgn\left( u_{i}\right) \nabla u_{i}\cdot \nabla \varphi _{1}dx \\
& +\int_{\Gamma }a_{i}\left( u_{i}\right) \partial _{\mathbf{n}%
}u_{i}\left\vert u_{i}\right\vert ^{m_{i}-1}sgn\left( u_{i}\right) \varphi
_{1}dS,
\end{align*}%
we deduce the following inequality%
\begin{eqnarray}
&&\partial _{t}\int_{\overline{\Omega }}\mathcal{E}\left( \overrightarrow{u}%
\left( t\right) \right) \varphi _{1}d\mu +\sum\nolimits_{i\in \mathbb{N}%
_{m}}\left( m_{i}-1\right) \int_{\Omega }a_{i}\left( u_{i}\right) \left\vert
\nabla u_{i}\right\vert ^{2}\left\vert u_{i}\right\vert ^{m_{i}-2}\varphi
_{1}dx  \label{p3} \\
&&+\sum\nolimits_{i\in \mathbb{N}_{m}}\int_{\Omega }a_{i}\left( u_{i}\right)
\left\vert u_{i}\right\vert ^{m_{i}-1}sgn\left( u_{i}\right) \nabla
u_{i}\cdot \nabla \varphi _{1}dx  \notag \\
&\leq &c\int_{\overline{\Omega }}\mathcal{E}\left( \overrightarrow{u}\left(
t\right) \right) \varphi _{1}d\mu +c.  \notag
\end{eqnarray}%
Here we have employed (\ref{A3}) to estimate the terms involving $%
f_{i},g_{i} $ in (\ref{p1})-(\ref{p2}). Let us now estimate the third
integral term on the left-hand side of (\ref{p3}). Exploiting the assumption
(\ref{A1}) on $a_{i}$, $i\in \mathbb{N}_{m}$, we deduce for $u=u_{i}$, $%
p=p_{i},$ $M=m_{i}$ that%
\begin{align}
\int_{\Omega }a\left( u\right) \left\vert u\right\vert ^{p-1}sgn\left(
u\right) \nabla u\cdot \nabla \varphi _{1}dx& \geq \int_{\Omega }\left\vert
u\right\vert ^{p+M-1}sgn\left( u\right) \nabla u\cdot \nabla \varphi _{1}dx
\label{p4} \\
& =\int_{\Omega }\nabla \overline{a}\left( u\right) \cdot \nabla \varphi
_{1}dx,  \notag
\end{align}%
where we have set%
\begin{equation*}
\overline{a}\left( u\right) :=\int_{0}^{\left\vert u\right\vert }a\left(
s\right) \left\vert s\right\vert ^{p-1}ds\geq c\left\vert u\right\vert
^{M+p}.
\end{equation*}%
Integrating by parts in (\ref{p4}) once more and noting that $\overline{a}%
\left( 0\right) =0$, we obtain%
\begin{align}
\int_{\Omega }\nabla \overline{a}\left( u\right) \cdot \nabla \varphi
_{1}dx& =\int_{\Gamma }\overline{a}\left( u\right) \partial _{\mathbf{n}%
}\varphi _{1}dS-\int_{\Omega }\overline{a}\left( u\right) \Delta \varphi
_{1}dx  \label{p5} \\
& =\int_{\Gamma _{1}}\overline{a}\left( u\right) \partial _{\mathbf{n}%
}\varphi _{1}dS+\int_{\Gamma _{2}}\overline{a}\left( u\right) \partial _{%
\mathbf{n}}\varphi _{1}dS  \notag \\
& -\int_{\Omega }\overline{a}\left( u\right) \Delta \varphi _{1}dx  \notag \\
& =\int_{\Gamma _{1}}\overline{a}\left( u\right) \Lambda _{1}\varphi
_{1}dS+\int_{\Omega }\overline{a}\left( u\right) \Lambda _{1}\varphi _{1}dx
\notag \\
& \geq \Lambda _{1}\int_{\overline{\Omega }}\left\vert u\right\vert
^{M+p}\varphi _{1}d\mu ,  \notag
\end{align}%
since $\left( \Lambda _{1},\varphi _{1}\right) $ satisfies (\ref{sp1})-(\ref%
{sp3}). Inserting the above estimates in (\ref{p3}), we get the following
inequality%
\begin{align}
& \partial _{t}\int_{\overline{\Omega }}\mathcal{E}\left( \overrightarrow{u}%
\left( t\right) \right) \varphi _{1}d\mu +\Lambda _{1}\sum\nolimits_{i\in
\mathbb{N}_{m}}\int_{\overline{\Omega }}\left\vert u_{i}\right\vert
^{m_{i}+p_{i}}\varphi _{1}d\mu  \label{p6} \\
& \leq c\int_{\overline{\Omega }}\mathcal{E}\left( \overrightarrow{u}\left(
t\right) \right) \varphi _{1}d\mu +c.  \notag
\end{align}%
Since all $p_{i}$'s are positive, we can absorb the integral term on the
right-hand side of (\ref{p6}), using the Young's inequality as follows:%
\begin{equation*}
\sum\nolimits_{i\in \mathbb{N}_{m}}\left\vert u_{i}\right\vert ^{m_{i}}\leq
\varepsilon \sum\nolimits_{i\in \mathbb{N}_{m}}\left\vert u_{i}\right\vert
^{m_{i}+p_{i}}+C_{\varepsilon },
\end{equation*}%
for a sufficiently small $\varepsilon \in \left( 0,\Lambda _{1}\right) $ and
some positive constant $C_{\varepsilon },$ independent of $u_{i},t.$
Moreover, setting $\nu =\min_{i\in \mathbb{N}_{m}}\left( m_{i}/p_{i}\right)
+1>1$, we immediatelly have from the above inequality, that%
\begin{equation}
\int_{\overline{\Omega }}\left( \mathcal{E}\left( \overrightarrow{u}\left(
t\right) \right) \right) ^{\nu }\varphi _{1}d\mu \leq c\sum\nolimits_{i\in
\mathbb{N}_{m}}\int_{\overline{\Omega }}\left\vert u_{i}\right\vert
^{m_{i}+p_{i}}\varphi _{1}d\mu +c,  \label{p7}
\end{equation}%
for some positive constant $c$ independent of \thinspace $\overrightarrow{u}$%
, $t$ and initial data. Using (\ref{p7}), we see that (\ref{p6}) yields the
following estimate%
\begin{equation}
\partial _{t}\int_{\overline{\Omega }}\mathcal{E}\left( \overrightarrow{u}%
\left( t\right) \right) \varphi _{1}d\mu +c\int_{\overline{\Omega }}\left(
\mathcal{E}\left( \overrightarrow{u}\left( t\right) \right) \right) ^{\nu
}\varphi _{1}d\mu \leq c,  \label{p8}
\end{equation}%
by an appropriate choice of $\varepsilon \leq \Lambda _{1}/2.$ By
normalizing the eigenfunction $\varphi _{1}$ in (\ref{p8}) such that $%
\left\Vert \varphi _{1}\right\Vert _{L^{1}\left( \overline{\Omega },d\mu
\right) }=1$, on account of Jensen's inequality, it follows that%
\begin{equation*}
\left( \int_{\overline{\Omega }}\mathcal{E}\left( \overrightarrow{u}\left(
t\right) \right) \varphi _{1}d\mu \right) ^{\nu }\leq \int_{\overline{\Omega
}}\left( \mathcal{E}\left( \overrightarrow{u}\left( t\right) \right) \right)
^{\nu }\varphi _{1}d\mu ,
\end{equation*}%
which gives the following estimate:%
\begin{equation}
\partial _{t}Y\left( t\right) +c_{1}\left( Y\left( t\right) \right) ^{\nu
}\leq c_{2},  \label{p9}
\end{equation}%
for some positive constants $c_{1},c_{2}$, where we have set%
\begin{equation*}
Y\left( t\right) :=\int_{\overline{\Omega }}\mathcal{E}\left(
\overrightarrow{u}\left( t\right) \right) \varphi _{1}d\mu .
\end{equation*}%
We can now use the Gronwall's inequality (see, e.g., \cite[Chapter III,
Lemma 5.1]{T}), applied to (\ref{p9}) to deduce that%
\begin{equation}
Y\left( t\right) \leq \left( \frac{c_{2}}{c_{1}}\right) ^{\nu }+\left(
c_{1}\left( \nu -1\right) t\right) ^{-\frac{1}{\nu -1}},\text{ }\forall t>0,
\label{p10}
\end{equation}%
which yields the desired claim. The proof of the theorem is complete.

\begin{remark}
\label{remi}In the case when $\overrightarrow{u}_{0}\in \mathcal{X}^{%
\overrightarrow{r}}$, $Y\left( 0\right) =\lim_{t\rightarrow 0^{+}}Y\left(
t\right) $ is finite, so a similar argument to \cite[Chapter III, Lemma 5.1]%
{T} gives%
\begin{equation*}
Y\left( t\right) \leq \max \left\{ Y\left( 0\right) ,\left( \frac{c_{2}}{%
c_{1}}\right) ^{\nu }\right\} ,\text{ }\forall t\geq 0.
\end{equation*}%
Thus, the second assertion in Theorem \ref{diss} also follows. We also note
that the above argument relies entirely on the fact that the boundary $%
\Gamma _{2}$ has positive measure and this gives $\Lambda _{1}>0$. The proof
fails to work if for instance, $\Gamma \equiv \Gamma _{1}$ (i.e., when $%
\Gamma _{2}=\varnothing $). We shall require different arguments for this
case (see below). Finally, if at least one $p_{i}=0$, for some $i\in \mathbb{%
N}_{m}$, the above argument can still be used to derive the following bound%
\begin{equation*}
\sup_{t\geq 0}\left\Vert \overrightarrow{u}\left( t\right) \right\Vert _{%
\mathcal{X}^{\overrightarrow{r}}}\leq Q\left( \left\Vert \overrightarrow{u}%
_{0}\right\Vert _{\mathcal{X}^{\overrightarrow{r}}}e^{ct}\right) .
\end{equation*}%
Indeed, this follows from a standard application of Gronwall's inequality to
(\ref{p6}).
\end{remark}

We now continue with the proof of Theorem \ref{diss2}. As in the proof of
Theorem \ref{diss}, we multiply (\ref{s1}) by $\left\vert u_{i}\right\vert
^{m_{i}-1}sgn\left( u_{i}\right) ,$ and integrate over $\Omega $, for each $%
i\in \mathbb{N}_{m}$. Then, we multiply both equations (\ref{s3}) and (\ref%
{s2}) by $\left\vert u_{i}\right\vert ^{m_{i}-1}sgn\left( u_{i}\right) $ and
integrate the relations that we obtain over $\Gamma $. Analogously to (\ref%
{p1})-(\ref{p2}) and arguing in a standard way as in (\ref{eqn3tris}), we
deduce the following identity:
\begin{align}
& \frac{d}{dt}\sum\nolimits_{i\in \mathbb{N}_{m}}\frac{1}{m_{i}}\left(
\int_{\Omega }\left\vert u_{i}\right\vert ^{m_{i}}dx+\int_{\Gamma
}\left\vert u_{i}\right\vert ^{m_{i}}dS\right)  \label{ba} \\
& +\sum\nolimits_{i\in \mathbb{N}_{m}}\left\langle f_{i}\left( x,t,%
\overrightarrow{u}\right) ,\left\vert u_{i}\right\vert ^{m_{i}-1}sgn\left(
u_{i}\right) \right\rangle _{L^{2}\left( \Omega \right) }  \notag \\
& +\sum\nolimits_{i\in \mathbb{N}_{m}}\left\langle g_{i}\left( x,t,%
\overrightarrow{y}\right) ,\left\vert u_{i}\right\vert ^{m_{i}-1}sgn\left(
u_{i}\right) \right\rangle _{L^{2}\left( \Gamma \right) }  \notag \\
& =-\sum\nolimits_{i\in \mathbb{N}_{m}}\left( m_{i}-1\right) \int_{\Omega
}a_{i}\left( x,t,\overrightarrow{u}\right) \left\vert \nabla
u_{i}\right\vert ^{2}\left\vert u_{i}\right\vert ^{m_{i}-2}dx.  \notag
\end{align}%
By assumption (\ref{A1}), we can estimate the term on the right-hand side of
(\ref{ba}) as follows:%
\begin{align}
\int_{\Omega }a_{i}\left( t,\overrightarrow{u}\right) \left\vert \nabla
u_{i}\right\vert ^{2}\left\vert u_{i}\right\vert ^{m_{i}-2}dx& \geq \alpha
_{i}\int_{\Omega }\left\vert \nabla u_{i}\right\vert ^{2}\left\vert
u_{i}\right\vert ^{m_{i}+p_{i}-2}dx  \label{ba2} \\
& =\alpha _{i}\left( \frac{2}{m_{i}+p_{i}}\right) ^{2}\int_{\Omega
}\left\vert \nabla \left\vert u\right\vert ^{\frac{m_{i}+p_{i}}{2}%
}\right\vert ^{2}dx.  \notag
\end{align}%
To estimate the nonlinear terms on the left-hand side of (\ref{ba}), we may
exploit (\ref{A3bis}). On account of (\ref{ba2}), we have
\begin{align}
& \frac{d}{dt}\sum\nolimits_{i\in \mathbb{N}_{m}}\frac{1}{m_{i}}\left(
\int_{\Omega }\left\vert u_{i}\right\vert ^{m_{i}}dx+\int_{\Gamma
}\left\vert u_{i}\right\vert ^{m_{i}}dS\right)  \label{ba3} \\
& +\sum\nolimits_{i\in \mathbb{N}_{m}}\alpha _{i}\left( \frac{2}{m_{i}+p_{i}}%
\right) ^{2}\left( m_{i}-1\right) \int_{\Omega }\left\vert \nabla \left\vert
u\right\vert ^{\frac{m_{i}+p_{i}}{2}}\right\vert ^{2}dx  \notag \\
& -\sum\nolimits_{i\in \mathbb{N}_{m}}\left( C_{f_{i}}\int_{\Omega
}\left\vert u_{i}\right\vert ^{m_{i}+p_{i}}dx+C_{g_{i}}\int_{\Gamma
}\left\vert u_{i}\right\vert ^{m_{i}+p_{i}}dS\right)  \notag \\
& \leq c.  \notag
\end{align}%
By assumption (\ref{eigen}), it follows that, for all $i\in \mathbb{N}_{m}$,
it holds%
\begin{align}
& a_{i}\left\Vert \nabla \varphi _{i}\right\Vert _{L^{2}\left( \Omega
\right) }^{2}-C_{f_{i}}\left\Vert \varphi _{i}\right\Vert _{L^{2}\left(
\Omega \right) }^{2}-C_{g_{i}}\left\Vert \varphi _{i}\right\Vert
_{L^{2}\left( \Gamma \right) }^{2}  \label{ba4} \\
& \geq \Lambda _{1,i}\left( \left\Vert \varphi _{i}\right\Vert _{L^{2}\left(
\Omega \right) }^{2}+\left\Vert \varphi _{i}\right\Vert _{L^{2}\left( \Gamma
\right) }^{2}\right) ,  \notag
\end{align}%
for all $\varphi _{i}\in H^{1}\left( \Omega \right) $. Thus, by choosing $%
\varphi _{i}=\left\vert u_{i}\right\vert ^{\left( m_{i}+p_{i}\right) /2}$ in
(\ref{ba4}), and recalling (\ref{energy}), from (\ref{ba3}), we obtain the
following inequality:%
\begin{align}
& \partial _{t}\int_{\overline{\Omega }}\mathcal{E}\left( \overrightarrow{u}%
\left( t\right) \right) d\mu +\Lambda _{1}\sum\nolimits_{i\in \mathbb{N}%
_{m}}\left( \int_{\Omega }\left\vert u_{i}\right\vert
^{m_{i}+p_{i}}dx+C_{g_{i}}\int_{\Gamma }\left\vert u_{i}\right\vert
^{m_{i}+p_{i}}dS\right)  \label{ba5} \\
& \leq c.  \notag
\end{align}%
Let us assume that $p_{i}>0$, for all $i\in \mathbb{N}_{m}$. Arguing now as
in the proof of Theorem \ref{diss} (see (\ref{p6})-(\ref{p9})), it is not
hard to see that we arrive at the following inequality%
\begin{equation}
\partial _{t}\int_{\overline{\Omega }}\mathcal{E}\left( \overrightarrow{u}%
\left( t\right) \right) d\mu +c\left( \int_{\overline{\Omega }}\mathcal{E}%
\left( \overrightarrow{u}\left( t\right) \right) d\mu \right) ^{v}\leq c,
\label{ba7}
\end{equation}%
where $v=\min_{i\in \mathbb{N}_{m}}\left( m_{i}/p_{i}\right) +1>1.$ Thus, we
can use the Gronwall's inequality as before (see (\ref{p10})) to derive the
estimate%
\begin{equation}
\int_{\overline{\Omega }}\mathcal{E}\left( \overrightarrow{u}\left( t\right)
\right) d\mu \leq c\left( 1+t^{-\frac{1}{\nu -1}}\right) ,\text{ }\forall
t>0.  \label{ba7b}
\end{equation}%
Hence, the first claim of the theorem follows from (\ref{ba7b}). On the
other hand, if at least one $p_{i}=0$ for some $i\in \mathbb{N}_{m},$ we
obtain the following analogue of (\ref{ba7}):%
\begin{equation}
\partial _{t}\int_{\overline{\Omega }}\mathcal{E}\left( \overrightarrow{u}%
\left( t\right) \right) d\mu +c\int_{\overline{\Omega }}\mathcal{E}\left(
\overrightarrow{u}\left( t\right) \right) d\mu \leq c,  \label{ba8}
\end{equation}%
which yields the second claim of the theorem once more on account of
Gronwall's inequality. In particular, there exists a positive function $Q$,
independent of initial data and time, such that%
\begin{equation}
\sup_{t\geq 0}\left\Vert \overrightarrow{u}\left( t\right) \right\Vert _{%
\mathcal{X}^{\overrightarrow{r}}}\leq Q\left( \left\Vert \overrightarrow{u}%
_{0}\right\Vert _{\mathcal{X}^{\overrightarrow{r}}}\right) e^{-c_{0}t}+C_{0},
\label{d3}
\end{equation}%
for some positive constants $c_{0},C_{0}$ independent of initial data and
time. The proof of Theorem \ref{diss2} is now complete.

\section{Appendix}

We will consider a more general problem than (\ref{1.11bb})-(\ref{1.12bb})
by taking $f_{1}\left( s\right) =f\left( s\right) -\lambda s$, $g_{1}\left(
s\right) =g\left( s\right) -\gamma s,$ provided that $\lambda ,\gamma >0$
are sufficiently large, and%
\begin{equation*}
f\left( s\right) \geq -c_{f}\text{, }g\left( s\right) \geq -c_{g}\text{, }%
\forall s\in \mathbb{R}\text{,}
\end{equation*}%
for some positive constants $c_{f},c_{g}.$ If $f\left( s\right) \sim
C_{f}\left\vert s\right\vert ^{p}s$ and $g\left( s\right) \sim
C_{g}\left\vert s\right\vert ^{q}s$, as $\left\vert s\right\vert \rightarrow
\infty $, for $p,q>1$ and some positive constants $C_{f},C_{g}$, it is
well-known \cite{G0} that problem%
\begin{equation}
\partial _{t}u-\nu \Delta u+f\left( u\right) -\lambda u=0,\text{ in }\Omega
\times (0,+\infty ),  \label{a1}
\end{equation}%
subject to the dynamic condition%
\begin{equation}
\partial _{t}u+\nu \partial _{\mathbf{n}}u+g\left( u\right) -\gamma u=0,%
\text{ on }\Gamma \times \left( 0,\infty \right) ,  \label{a2}
\end{equation}%
and initial condition%
\begin{equation}
u_{\mid t=0}=u_{0}\text{ in }\overline{\Omega },  \label{a3}
\end{equation}%
possesses a finite dimensional global attractor $\mathcal{A}_{\text{dyn}}$
which is bounded in $H^{2}\left( \Omega \right) \cap \mathbb{X}^{\infty }$.

Let $u_{\ast }$ be a constant (hyperbolic)\ equilbrium for the system (\ref%
{a1})-(\ref{a2}) (see \cite[Section 3]{GalNS}). We linearize (\ref{a1})-(\ref%
{a2}) around $u_{\ast }.$ We obtain%
\begin{equation}
\partial _{t}u=\nu \Delta u-\left( f^{^{\prime }}\left( u_{\ast }\right)
-\lambda \right) u,\text{ in }\Omega \times (0,+\infty ),  \label{la1}
\end{equation}%
subject to the dynamic condition%
\begin{equation}
\partial _{t}u=-\nu \partial _{\mathbf{n}}u-\left( g^{^{\prime }}\left(
u_{\ast }\right) -\gamma \right) u,\text{ on }\Gamma \times \left( 0,\infty
\right) .  \label{la2}
\end{equation}%
We aim to better understand the nature of the (invariant)\ unstable
eigenspace $E^{u}$ which corresponds to the following (matrix) operator%
\begin{equation*}
\mathbf{L}\left( u_{\ast }\right) W=\binom{\nu \Delta w-f^{^{\prime }}\left(
u_{\ast }\right) w+\lambda w}{-\nu \partial _{\mathbf{n}}w-g^{^{\prime
}}\left( u_{\ast }\right) w+\gamma w},\text{ }W=\left(
\begin{array}{c}
w \\
w_{\mid \Gamma }%
\end{array}%
\right) ,
\end{equation*}%
with $\sigma \left( \mathbf{L}\left( u_{\ast }\right) \right) \subset
\left\{ \zeta :\zeta >0\right\} $. We note that $\left( \mathbf{L}\left(
u_{\ast }\right) ,\text{dom}\left( \mathbf{L}\left( u_{\ast }\right) \right)
\right) $\ is self-adjoint on $\mathbb{X}^{2}$\ with spectrum contained in $%
\left( -\infty ,C_{\lambda ,\gamma }\right] ,$ for some $C_{\lambda ,\gamma
}>0$ which depends only on $f,g,\lambda $ and $\gamma $ (see, e.g., \cite%
{GalNS} and references therein). Next, let $\left\{ \varphi _{j}\left(
x\right) \right\} _{j\in \mathbb{N}_{0}}$ be an orthonormal basis in $%
\mathbb{X}^{2}$ consisting of eigenfunctions of the (positive) Wentzell
Laplacian $\Delta _{W}$ operator (see \cite[Theorem 5.1]{GalNS})%
\begin{equation}
\Delta _{W}\varphi _{j}=\Lambda _{j}\varphi _{j},\text{ }j\in \mathbb{N}_{0},%
\text{ }\varphi _{j}\in \text{dom}\left( \Delta _{W}\right) \cap C\left(
\overline{\Omega }\right)  \label{evseq}
\end{equation}%
such that%
\begin{equation*}
0=\Lambda _{0}<\Lambda _{1}\leq \Lambda _{2}\leq ...\leq \Lambda _{,j}\leq
\Lambda _{j+1}\leq ....\rightarrow +\infty .
\end{equation*}%
We shall seek for eigenvectors $W_{j}=\binom{w_{j}}{w_{j\mid \Gamma }}\in
\mathbb{X}^{2}$, of the form $w_{j}\left( x\right) =\varphi _{j}\left(
x\right) p_{j},$ $p_{j}\in \mathbb{R}$, satisfying equation%
\begin{equation}
\mathbf{L}\left( u_{\ast }\right) W_{j}=\zeta _{j}W_{j},\text{ }W_{j}\in
\text{dom}\left( \mathbf{L}\left( u_{\ast }\right) \right) :=\text{dom}%
\left( \Delta _{W}\right) .  \label{eee}
\end{equation}%
Note that for $W_{j}\in dom\left( \mathbf{L}\left( u_{\ast }\right) \right)
\subset H^{1}\left( \Omega \right) \times L^{2}\left( \Gamma \right) ,$ the
trace of $w_{j}$ makes sense as an element of $H^{1/2}\left( \Gamma \right) $%
. Substituting such $w_{j}$ into (\ref{eee}), taking into account (\ref%
{evseq}) and the fact that%
\begin{equation*}
\mathbf{L}\left( u_{\ast }\right) W_{j}=-\nu \Delta _{W}W_{j}+\Pi _{\lambda
,\gamma }W_{j},\text{ }\Pi _{\lambda ,\gamma }W_{j}:=\binom{(-f^{^{\prime
}}\left( u_{\ast }\right) +\lambda )w_{j}}{(-g^{^{\prime }}\left( u_{\ast
}\right) +\gamma )w_{j\mid \Gamma }},
\end{equation*}%
we obtain the equation%
\begin{equation}
\left( -\nu \Lambda _{j}I+\Pi _{\lambda ,\gamma }\right) p_{j}=\zeta
_{j}p_{j},\text{ }\Pi _{\lambda ,\gamma }=\left(
\begin{array}{cc}
-f^{^{\prime }}\left( u_{\ast }\right) +\lambda & 0 \\
0 & -g^{^{\prime }}\left( u_{\ast }\right) +\gamma%
\end{array}%
\right) .  \label{pj}
\end{equation}%
A nonzero $p_{j}$ exists if $\zeta =\zeta _{j}$ is a root of the equation%
\begin{equation}
\det \left( -\nu \Lambda _{j}I+\Pi _{\lambda ,\gamma }-\zeta I\right) =0,%
\text{ }\zeta >0.  \label{dett}
\end{equation}%
When $\nu =0,$ this equation has at least one root $\zeta >0$ provided that
at least one of $\lambda $ and $\gamma $ is sufficiently large, i.e., either
$\lambda >f^{^{\prime }}\left( u_{\ast }\right) $ or $\gamma >g^{^{\prime
}}\left( u_{\ast }\right) $ (in fact the roots are $\zeta =\lambda
-f^{^{\prime }}\left( u_{\ast }\right) $ and $\zeta =\gamma -g^{^{\prime
}}\left( u_{\ast }\right) $, respectively). Therefore, there exists $\delta
>0$ such that when $\nu \Lambda _{j}<\delta ,$ the equation (\ref{dett}) has
a root $\zeta _{j}\left( \mathbf{L}\right) =\zeta _{j}\left( \nu \right) $
with $\zeta _{j}>0$. Therefore, to any such root $\zeta _{j}$, we can assign
a nontrivial $p_{j},$ which is a solution of (\ref{pj}), and thus an
eigenvector $W_{j}$. Let us now compute how many $j$'s satisfy the
inequality $\nu \Lambda _{j}<\delta $. When $N\geq 3$, the asymptotic
behavior of $\Lambda _{j}$ is
\begin{equation}
\Lambda _{j}\sim C_{S}\left( \Gamma \right) j^{1/\left( N-1\right) }\text{
as }j\rightarrow \infty  \label{a5}
\end{equation}%
(see \cite[Theorem 5.4]{GalNS}). The inequality $\nu \Lambda _{j}<\delta $
certainly holds when%
\begin{equation}
1\leq j\leq C_{\lambda ,\gamma }\delta ^{n-1}\left( C_{S}\left( \Gamma
\right) \nu \right) ^{1-N}=C_{\lambda ,\gamma }^{^{\prime }}\left\vert
\Gamma \right\vert \left( \frac{1}{\nu }\right) ^{N-1},\text{ for }N\geq 3,
\label{number}
\end{equation}%
where the positive constants $C_{\lambda ,\gamma },C_{\lambda ,\gamma
}^{^{\prime }}$ depend only on $\lambda ,\gamma $ and $N$.

\begin{remark}
Note that the number of unstable mode solutions to (\ref{la1})-(\ref{la2})
obeys the same relation (\ref{number}) even when $f\equiv 0$ and $\lambda =0$
in (\ref{a1}) (i.e., the dynamics of $u$ inside the bulk $\Omega $ is
strictly linear). Finally, we note that both $C_{\lambda ,\gamma
},C_{\lambda ,\gamma }^{^{\prime }}\rightarrow +\infty $ if either $\gamma
\rightarrow +\infty $ or $\lambda \rightarrow +\infty $ (cf. also \cite[%
Section 3]{GalNS}). In this case the instability index of $u_{\ast }$ is%
\begin{equation*}
N_{+}\left( u_{\ast }\right) \sim C_{\lambda ,\gamma }^{^{\prime
}}\left\vert \Gamma \right\vert \left( \frac{1}{\nu }\right) ^{N-1},\text{ }%
N\geq 3.
\end{equation*}
\end{remark}


\begin{thebibliography}{99}
\bibitem{AIMT} F. Andreu, N. Igbida, J. M. Maz\'{o}n, J. Toledo,
Renormalized solutions for degenerate elliptic-parabolic problems with
nonlinear dynamical boundary conditions and $L^{1}$-data, J. Differential
Equations 244 (2008), 2764--2803.

\bibitem{AIMT2} F. Andreu, N. Igbida, J. M. Maz\'{o}n, J. Toledo, A
degenerate elliptic-parabolic problem with nonlinear dynamical boundary
conditions, Interfaces Free Bound. 8 (2006), 447--479.

\bibitem{Be2} J. von Below, M.G. Mailly, Blow up for reaction diffusion
equations under dynamical boundary conditions, Comm. Partial Differential
Equations 28 (2003), no. 1-2, 223--247.

\bibitem{BF} C. Bandle, J. Von Below, W. Reichel, Parabolic problems with
dynamical boundary conditions: eigenvalue expansions and blow up, Rend.
Lincei Mat. Appl. 17 (2006), 35-67.

\bibitem{BV} A.~V. Babin and M.~I. Vishik, Attractors of evolutionary
equations, Nauka, Moscow, 1989.

\bibitem{BFP} I. Borsi, A. Farina, M. Primicerio, A rain water infiltration
model with unilateral boundary condition: qualitative analysis and numerical
simulations, Math. Methods Appl. Sci. 29 (2006), 2047--2077.

\bibitem{CGGM} C. Cavaterra, C.G. Gal, M. Grasselli, A. Miranville,
Phase-field systems with nonlinear coupling and dynamic boundary conditions,
Nonlinear Anal. 72 (2010), no. 5, 2375--2399.

\bibitem{CJ} A. Constantin, J. Escher, Global existence for fully parabolic
boundary value problems, NoDEA Nonlinear Differential Equations Appl. 13
(2006), no. 1, 91--118.

\bibitem{Co} C. Cosner, Reaction-diffusion equations and ecological
modeling, Tutorials in mathematical biosciences, IV, Lecture Notes in Math.,
1922, 77--115 Springer, Berlin, 2008.

\bibitem{DiB} E. Dibenedetto, Degenerate Parabolic equations, Universitext,
Springer-Verlag, 1993.

\bibitem{Du} L. Dung, Global attractors and steady state solutions for a
class of reaction-diffusion systems, J. Differential Equations 147 (1998),
no. 1, 1--29.

\bibitem{Du2} L. Dung, H\"{o}lder regularity for certain strongly coupled
parabolic systems, J. Differential Equations 151 (1999), no. 2, 313--344.

\bibitem{Du3} L. Dung, Dissipativity and global attractor for a class of
quasilinear parabolic systems, Comm. Partial Differntial Equations 22
(1997), pp. 413--433.

\bibitem{EPS} E.A. Ermakova, M.A. Panteleev, E.E. Schnol, Blood coagulation
and propagation of autowaves in flow, Pathophysiology of Haemostasis and
Thrombosis, Vol. 34, pp. 135-142, 2005.

\bibitem{EZ} M. Efendiev, S. Zelik, Finite and infinite-dimensional
attractors for porous media equations, Proc. London Math. Soc. (3) 96 (2008)
51--77.

\bibitem{Escher} J. Escher, \textit{Quasilinear parabolic systems with
dynamical boundary conditions}, Comm. Partial Differential Equations, 18
(1993), 1309--1364.

\bibitem{Escher2} J. Escher, \textit{On quasilinear fully parabolic boundary
value problems}, Differential Integral Equations, 7 (1994), 1325--1343.

\bibitem{Evans} L. C. Evans, Partial differential equations, Graduate
Studies in Mathematics 19, American Mathematical Society, Providence, RI,
2010.

\bibitem{FL} J. Filo, S. Luckhaus, Modelling surface runoff and infiltration
of rain by an elliptic-parabolic equation coupled with a first-order
equation on the boundary, Arch. Rational Mech. Anal. 146 (1999) 157--182.

\bibitem{FH} J.Z. Farkas, P. Hinow, Physiologically structured populations
with diffusion and dynamic boundary conditions, Mathematical Biosciences and
Engineering, Vol.8, (2011), 503-513.

\bibitem{FQ} M. Fila, P. Quittner, Large time behavior of solutions of a
semilinear parabolic equation with a nonlinear dynamical boundary condition,
Topics in nonlinear analysis, 251--272, Progr. Nonlinear Differential
Equations Appl., 35, Birkh\"{a}user, Basel, 1999.

\bibitem{GalNS} C.~G.~Gal, Sharp estimates for the global attractor of
scalar reaction-diffusion equations with a Wentzell boundary condition, J.
Nonlinear Science 22 (2012), 85-106.

\bibitem{G0} C. G. Gal, On a class of degenerate parabolic equations with
dynamic boundary conditions, Journal of Differential Equations 253 (2012),
126--166.

\bibitem{GV} G. Galiano, J. Velasco, A dynamic boundary value problem
arising in the ecology of mangroves, Nonlinear Anal. Real World Appl. 7
(2006), 1129--1144.

\bibitem{Gi} G. R. Goldstein, Derivation and physical interpretation of
general boundary conditions, Adv. Differential Equations 11 (2006), 457--480.

\bibitem{GW} C. G. Gal, M. Warma, Well-posedness and the global attractor of
some quasilinear parabolic equations with nonlinear dynamic boundary
conditions, Differential Integral Equations, 23 (2010), 327-358.

\bibitem{LSU} O. A. Ladyzenskaya, V. A. Solonnikov, N. N. Ural'tseva, Linear
and quasilinear equations of parabolic type, AMS Translations Monograph, 23
(1968).

\bibitem{HH} J.B. Haun, D.A. Hammer, Quantifying nanoparticle adhesion
mediated by specific molecular interactions, Langmuir, Vol. 24, 2008.

\bibitem{Ig} N. Igbida, Hele-Shaw type problems with dynamical boundary
conditions, J. Math. Anal. Appl. 335 (2007), 1061--1078.

\bibitem{Ig2} N. Igbida, M. Kirane, A degenerate diffusion problem with
dynamical boundary conditions, Math. Ann. 323 (2002), 377--396.

\bibitem{Ma} M. Meyries, Maximal regularity in weighted spaces, nonlinear
boundary conditions, and global attractors, PhD thesis, 2010.

\bibitem{MPE} N.M. Andersen, M.P. Sorensen, M.A. Efendiev, O.H. Olsen, S.H.
Ingwersen, Modelling of the Blood Coagulation Cascade in an In Vitro Flow
System, In: International Journal of Biomathematics and Biostatistics, vol
1, (2010), 1-7.

\bibitem{M} D. Mugnolo, Vector-valued heat equations and networks with
coupled dynamic boundary conditions, Adv. Differential Equations 15, (2010),
1125--1160.

\bibitem{RB} A. R. Bernal, A. Tajdine, Nonlinear balance for
reaction-diffusion equations under nonlinear boundary conditions:
dissipativity and blow-up, Journal of Differential Equations 169, (2001),
332-372.

\bibitem{ST} R. Showalter, T. D. Little, U. Hornung, Parabolic PDE with
hysteresis, Control Cybernet. 25 (1996), 631--643.

\bibitem{Su} N. Su, Multidimensional degenerate diffusion problem with
evolutionary boundary condition: existence, uniqueness and approximation,
Flow in porous media (Oberwolfach, 1992), 165--178, Internat. Ser. Numer.
Math., 114, Birkh\"{a}user, Basel, 1993.

\bibitem{T} R.~Temam, \textquotedblleft Infinite-Dimensional Dynamical
Systems in Mechanics and Physics," Springer-Verlag, New York, 1997.

\bibitem{Va} J. L. V\'{a}zquez, The porous medium equation: Mathematical
theory, Oxford Mathematical Monographs, The Clarendon Press, Oxford
University Press, Oxford, 2007.
\end{thebibliography}
\end{document}